\documentclass[12pt,a4paper]{ouparticle}
\usepackage{geometry}
\usepackage{amssymb}
\usepackage{euscript}
\RequirePackage{setspace}
\usepackage[linkcolor=green!40!blue,colorlinks=green!40!blue]{hyperref} 
\hypersetup{colorlinks=true, citecolor=blue, filecolor=Maroon,
 linkcolor=blue, urlcolor=black!65!purple}

\usepackage{xcolor,colortbl}

\definecolor{Gray}{gray}{0.35}
\definecolor{LightCyan}{rgb}{0.88,1,1}

 \usepackage{amsthm}
\usepackage[OT2,T1]{fontenc}

\newcommand{\F}{\mathcal{F}}

\newcommand{\I}{\mathcal{I}}

\newcommand{\D}{\mathcal{D}}

\newcommand{\A}{\mathcal{A}}

\newcommand{\PS}{\mathbf{PreStk}}
\newcommand{\Q}{\mathbf{QCoh}}

\newcommand{\M}{\mathcal{M}}

\newcommand{\EQ}{\mathcal{Y}}

\newcommand{\RS}{{\color{white!05!black}\mathbb{R}\underline{\mathrm{Sol}}}}

\newcommand{\Ld}{{\color{white!10!black}\mathbf{L}}}

\newcommand{\DG}
{{\color{white!10!black}\mathbf{dg}_{\mathcal{D}_X}}}

\usepackage{tikz-cd}
\tikzset{%
    symbol/.style={%
        draw=none,
        every to/.append style={%
            edge node={node [sloped, allow upside down, auto=false]{$#1$}}}
    }
}
\usepackage{adjustbox}
\numberwithin{equation}{section}

\usepackage{aliascnt}
\usepackage{hyperref, cleveref}
\newtheorem{thm}{Theorem}[section]
\newaliascnt{lem}{thm}

\aliascntresetthe{lem}

\newaliascnt{sublem}{thm}
\newtheorem{sublem}[sublem]{Sub-lemma}
\aliascntresetthe{sublem}

\newaliascnt{warn}{thm}

\aliascntresetthe{warn}

\newaliascnt{ass}{thm}
\newtheorem{ass}[ass]{Assumptions}
\aliascntresetthe{ass}

\newaliascnt{prop}{thm}
\newtheorem{prop}[prop]{Proposition}
\aliascntresetthe{prop}

\newaliascnt{cor}{thm}
\newtheorem{cor}[cor]{Corollary}
\aliascntresetthe{cor}

\newaliascnt{defn}{thm}
\newtheorem{defn}[defn]{Definition}
\aliascntresetthe{defn}

\newaliascnt{ex}{thm}
\newtheorem{ex}[ex]{Example}
\aliascntresetthe{ex}
\usepackage{physics}
\newaliascnt{obs}{thm}
\newtheorem{obs}[obs]{Observation}
\aliascntresetthe{obs}
\usepackage{tikz-cd}
\tikzset{%
    symbol/.style={%
        draw=none,
        every to/.append style={%
            edge node={node [sloped, allow upside down, auto=false]{$#1$}}}
    }
}
\usepackage{adjustbox}

\newaliascnt{rmk}{thm}
\newtheorem{rmk}[rmk]{Remark}
\aliascntresetthe{rmk}

\newaliascnt{notate}{thm}
\newtheorem{notate}[notate]{Notation}
\aliascntresetthe{notate}

\newaliascnt{rem}{thm}
\newtheorem{rem}[rem]{Reminder}
\aliascntresetthe{rem}

\newaliascnt{coords}{thm}

\aliascntresetthe{coords}

\newaliascnt{interpret}{thm}

\aliascntresetthe{interpret}

\aliascntresetthe{ex}

\newaliascnt{term}{thm}

\aliascntresetthe{term}

\newaliascnt{infmdefn}{thm}

\aliascntresetthe{infmdefn}

\newaliascnt{cons}{thm}
\newtheorem{cons}[cons]{Construction}
\aliascntresetthe{cons}

\usepackage{authblk}

\AtEndDocument{%
  \par
  \medskip
 \flushleft \begin{tabular}{@{}l@{}}%
    \textsc{\small{Jacob Kryczka}}\\
    \textsc{\footnotesize{Beijing Institute of Mathematical Sciences and Applications (BIMSA)}}
    \\
    \scriptsize{\textit{E-mail}: \texttt{jkryczka@bimsa.cn}}
    \\
    \\
    \textsc{\small{Artan Sheshmani}}\\
    \textsc{\footnotesize{Beijing Institute of Mathematical Sciences and Applications (BIMSA)}}
    \\
    \textsc{\footnotesize{Massachusetts Institute of Technology (MIT), IAiFi Institute} }
\\
\textsc{\footnotesize{Laboratory of Mirror Symmetry, Higher School of Economics (NRU HSE)}}
\\
   \scriptsize{\textit{E-mail}: \texttt{artan@mit.edu}}
  \end{tabular}}

\begin{document}

\title{The $\D$-Geometric Hilbert Scheme -- Part II:\\
\Large{Hilbert and Quot DG-Schemes}}

\author{%
\name{\normalsize{Jacob Kryczka and Artan Sheshmani}}}

\abstract{This is the second in a series of two papers developing a moduli-theoretic framework for differential ideal sheaves associated with formally integrable, involutive systems of algebraic partial differential equations (PDEs). Building on earlier work, which established the existence of moduli stacks for such systems with prescribed regularity and stability conditions, we now construct a derived enhancement of these moduli spaces. We prove the derived $\D$-Quot functor admits a global differential graded refinement representable by a suitable differential graded $\D$-manifold. We further analyze the finiteness, representability, and functoriality properties of these derived moduli spaces, establishing foundations for a derived deformation theory of algebraic differential equations.} 


\date{\today}


\maketitle
\tableofcontents

\section{Introduction}
\label{sec: Introduction}
A fundamental challenge in the geometric study of partial differential equations (PDEs) is the systematic construction of moduli spaces that classify formally integrable and involutive algebraic differential systems. These systems are naturally encoded as $\D_X$-ideal sheaves inside the algebra of infinite jets $\mathcal{O}(J_X^{\infty}E)$ associated with a vector bundle $E$ over a smooth variety (or complex manifold) $X$. Their structure is governed by symbolic invariants and numerical data, including a Hilbert-type polynomial that captures the asymptotic growth of solutions under prolongation.

This paper is a continuation of the approach introduced in \cite{KSh}, with the aim of recasting the geometric theory of PDEs in the spirit of \cite{KLV}, within a unifying moduli-theoretic perspective.
In the previous work the $\D$-Hilbert and $\D$-Quot functors we introduced, and their representability was proven. The proof made use of techniques present in Douady's treatment of moduli of complex analytic subspaces \cite{Dou}, Grothendieck's original construction of Hilbert schemes \cite{Gro3}, and ideas from GIT \cite{M,HL}, adapted to the setting of differential ideals.

The moduli problem is bounded upon introducing a cohomological 
stability criteria (Spencer-stability), which serves as an analogue of Gieseker stability in the context of algebraic PDEs. This condition, defined in terms of Spencer cohomology and the Hilbert polynomial of the symbol module, ensures the existence of well-behaved moduli spaces for certain differential systems.
Moreover, in the course of the proof, building on the theory of Lie-pseudogroups of  \cite{C2,E,Ku2,KuSp,SiSt}, we showed that algebraic pseudogroups act naturally on jet spaces and that their invariant quotients may be described as well-defined $\D$-schemes. A PDE analog of the Lefschetz hyperplane theorem, was proven that ensured that Spencer regularity is preserved under restriction to non-characteristic subvarieties. This allows us to control local-to-global properties of symbolic data. We identified a stable locus defined via Spencer cohomology, which is isomorphic to a finite-type scheme, interpreted as the coarse moduli space for (Spencer) semi-stable algebraic differential equations, and prove the main result which roughly speaking, states the following.
\\

\noindent\textbf{Theorem} (\cite{KSh}) \hspace{1mm}\emph{Fix a numerical polynomial $P\in \mathbb{Q}[t].$ Then there exists a moduli functor classifying formally integrable (non-singular) algebraic differential systems with fixed number of dependent and independent variables and with specified (Spencer) regular symbolic behavior whose numerical polynomial $P_{\D}$ is equal to $P.$ Moreover, there exists a sub-functor consisting of Spencer semi-stable differential ideals representable by an ind-finite-type $\D$-scheme.}
\\

In the present paper, we develop a derived enhancement of these moduli spaces, refining their structure to incorporate homotopical and higher categorical aspects. Building on the techniques of \cite{BKS,BKSY2}, we construct a simplicial $\D$-presheaf that parametrizes, up to homotopy $A_{\infty}$-module structures on characteristic modules (e.g. symbolic systems) within the algebra of jet schemes. This extends foundational work of Ciocan-Fontanine and Kapranov \cite{CFK,CFK2}, who introduced a dg-moduli functor for submodules via a scheme of differential graded ideals, adapting their framework to the setting of certain algebraic nonlinear PDEs.

\subsection{Main result and organization}
In Section \ref{sec: Derived D Geometry} we briefly recall the basic tools from derived algebraic $\D$-geometry, following \cite{KSY,KSY2}. In particular, we establish several key results concerning descent, base-change and $\D$-geometric flatness in this context.
In Section \ref{sec: Derived moduli of D-ideals} we construct a derived enhancement of the moduli functor for $\D$-ideal sheaves. We discuss a $\D$-geometric version of Grassmannians.

In Section \ref{sec: D-Quot schemes as DG-D-schemes} we pass from the derived setting to the differential graded context. A key motivation for this transition from derived algebraic geometry to a differential graded model lies in the desire for an explicit underlying scheme-theoretic structure  compatible with the simplicial dg-manifold representing the moduli problem (see e.g. \cite{BKS,BKSY2}).
We define some natural notions for a dg-version of $\D$-geometry, similar in spirit but inequivalent to the derived $\D$-geometry counterparts in \cite{KSY}, and prove the main result. We alwars consider those smooth algebraic varieties $X$ which are $D$-affine (see \cite[Section 1.3]{KSh} and \cite[Theorem 1.6.5]{HT}).
\\

\noindent\textbf{Theorem.} (Proposition \ref{prop: DerDQuot is stack}, Proposition \ref{prop: Z-diagram} and Theorem 
\ref{thm: Finite--type dg-D-Quot})
\emph{
Let $X$ be a smooth $\D$-affine algebraic variety of dimension $n$, let $E$ be a fixed vector bundle of rank $m$ and consider a fixed $\D$-smooth sub-scheme $Z$ of the scheme of algebraic jets $J_X^{\infty}E.$ Fix a numerical polynomial $P\in \mathbb{Q}[t].$ Then, there  is a simplicial $\D$-presheaf $\mathbb{R}\underline{\mathcal{Q}\mathrm{uot}}_{\mathcal{D}_X,Z}^{P},$ which is a $\mathcal{D}$-stack for the $\D$-étale topology. Moreover, for involutive and formally integrable $\D$-algebraic PDEs, it is equivalent to a dg-$\D$-stack paramaterizing up-to-homotopy sub-module structures in the associated characteristic module which possesses a sub-functor 
corresponding to finite-type PDEs which is representable by a dg-$\mathcal{D}_X$-manifold almost of finite type.}
\\

Consequently, by Proposition \ref{prop: DerDQuot is stack}, the classical truncation is isomorphic to $\underline{\mathrm{Hilb}}_{\D}^P(Z),$ the classical $\D$-Hilbert moduli constructed in \cite[Theorem 5.1]{KSh}. This suggest that the main result of this paper may be viewed as the PDE analog of \cite{CFK2}. 
That is, it gives the `differential' analog of the main result of \cite{BKSY2} stating that when $X$ is a projective scheme and $\mathcal{F}$ is a coherent sheaf on $X$, then the derived Quot functor $\mathbb{R}\underline{\mathcal{Q}uot}_{\mathcal{F}}(X)$ is a derived stack represented by a DG-manifold of finite-type.

The constructions and techniques discussed in this paper serve to bridge the abstract homotopy-theoretic approach to higher geometry with concrete geometric invariants associated with PDEs. This provides a robust foundation for the deformation theory and obstruction calculus of algebraic differential systems, indicating the possibility to obtain an ``enumerative algebraic $\D$-geometry,'' with evident applications to counting problems in gauge theory and complex geometry.
A more systematic study of these problems is however outside of the scope of the current paper.
\\

\textbf{Acknowledgments.} 
J.K is supported by the Postdoctoral International Exchange Program of
Beijing Municipal Human Resources and Social Security Bureau. A.S. is supported by grants from Beijing Institute of Mathematical Sciences and Applications (BIMSA), the Beijing NSF BJNSF-IS24005, and the China National Science Foundation (NSFC) NSFC-RFIS program W2432008. He would like to thank China's National Program of Overseas High Level Talent for generous support.

\subsection{Conventions and notations}
\label{Notations and Conventions}
Throughout this paper we treat non-linear PDEs modulo their symmetries in a coordinate free way using $\D$-modules and $\D$-algebras and we’ll mainly work with related $(\infty,1)$-categories. The category of commutative unital left $\D$-algebras is denoted $\mathrm{CAlg}_X(\mathcal{D}_X)$ and its $\infty$-categorical connective enhancement by $\mathbf{CAlg}(\mathcal{D}_X).$  

Bold-faced notation such as $\mathbf{C}$ will be reserved for indicating higher homotopical structure e.g. $\infty$-categorical, differential graded (dg) or model categorical enrichments of ordinary (classical) abelian categories $\mathrm{C}.$ Similarly for $\mathbf{Spec}.$

\begin{notate}
\label{notate: DAG}
\normalfont
\begin{itemize}
 \item $\mathbf{Spc}$ is the symmetric monoidal $(\infty,1)$-category of spaces and for a field $k$ of characteristic zero $\mathbf{Vect}_k$ is the stable symmetric monoidal $(\infty,1)$-category of unbounded cochain complexes.
    \item  $\PS_k$ is the category of prestacks i.e. functors $\mathrm{Fun}\big(\mathbf{CAlg}_k,\mathbf{Spc}),$ where $\mathbf{CAlg}_k:=\mathbf{CAlg}(\mathbf{Vect}_k)$ is the $(\infty,1)$-category of commutative algebras in $\mathbf{Vect}_k.$

    \item Denote by $(-)^{\simeq}:\mathbf{Cat}_{\infty}\rightarrow \mathbf{Spc},$ the functor which takes the maximal $\infty$-subgroupoid of an $\infty$-category.

    \item For any $(\infty,1)$-category $\mathbf{C}$ and an operad $\mathcal{O},$ we denote by 
$\mathbf{Alg}_{\mathcal{O}}[\mathbf{C}],$
the category of $\mathcal{O}$-algebras in $\mathbf{C},$ e.g. $\mathbf{CAlg}_{k}:=\mathbf{Alg}_{\mathcal{C}\mathrm{omm}}[\mathbf{Vect}_k].$

\item Global sections $\Gamma(-,\mathcal{O}):\mathbf{AffSch}_{S}^{op}\rightarrow \mathbf{QCoh}(S)^-$ is symmetric monoidal with respect to the co-Cartesian symmetric monoidal structure on $\mathbf{AffSch}_{S}^{op}$ and the usual tensor product on $\mathbf{QCoh}(S).$ 

\item $\mathbf{IndAffSch}_{S}\hookrightarrow \mathbf{PStk}_S$ is the $\infty$-category of ind-affine ind-schemes over $S$. Right Kan extension of global sections defines a functor
$$Kan^R_{\mathbf{AffSch}_{S}^{op}\hookrightarrow \mathbf{IndAffSch}_{S}^{op}}(\Gamma(-,\mathcal{O})):\mathbf{IndAffSch}_{S}^{op}\rightarrow \mathbf{Pro}\big(\mathbf{QCoh}(S)^-\big),$$
denoted $\Gamma(-,\mathcal{O})^{\mathrm{Top}}.$ This extends to arbitrary prestacks $Z$, in the usual way e.g. by noting that 
$\mathbf{IndAffSch}_{Z}\simeq \underset{S\rightarrow Z}{\mathrm{lim}}\hspace{1mm}\mathbf{IndAffSch}_{S}.$
\end{itemize}
\end{notate}
For a coherent $\D$-module with good filtration, recall $\mathrm{Gr}^F(\mathcal{M})$ is coherent graded $\mathrm{Gr}^F(\mathcal{D}_X)$-module. 
    
    Since $\mathrm{Gr}^F(\mathcal{D}_X)\subset \pi_*^{X}\mathcal{O}_{T^*X}$, via the monomorphism $\mathrm{Sym}_{\mathcal{O}_X}^*(\Theta_X)\hookrightarrow \pi_*^X\mathcal{O}_{T^*X}$
we have $\pi_X^{-1}\mathrm{Gr}^F(\mathcal{D}_X)\subset \mathcal{O}_{T^*X}$ and $\mathcal{O}_{T^*X}$ is faithfully flat over $\pi_X^{-1}\mathrm{Gr}^F(\mathcal{D}_X),$ with $\pi_X:T^*X\rightarrow X$ the canonical projection.

In `microlocal' coordinates $(x;\xi)$ on $T^*X,$ set $\mathcal{O}_{T^*X}(p)$ to be the sheaf of holomorphic functions on $T^*X$ which are $\xi$-polynomial and homogeneous of degree $p.$
In particular, we have canonical identifications
$$\mathrm{Gr}_0^{\mathrm{F}}\mathcal{D}_X\simeq \mathcal{O}_X\simeq \pi_*^X\big(\mathcal{O}_{T^*X}(0)\big),\hspace{2mm}\text{ and }\hspace{1mm} \mathrm{Gr}_1^{\mathrm{F}}\mathcal{D}_X\simeq \Theta_X\simeq \pi_*^X\big(\mathcal{O}_{T^*X}(0)\big),$$
as $\mathcal{O}_X$-modules.
There is an exact functor
$\widetilde{(-)}$ from $\mathrm{Gr}^F(\mathcal{D}_X)$-modules to $\mathcal{O}_{T^*X},$ modules that we define as $\widetilde{\mathcal{M}}:=\mathcal{O}_{T^*X}\otimes_{\pi_X^{-1}\mathrm{Gr}^F\mathcal{D}_X}\pi_X^{-1}M.$ By \cite[Thm 2.6.]{K} the support
$\mathrm{Supp}\big(\widetilde{\mathrm{Gr}^F(\mathcal{M}}\big)\subset T^*X,$
is independent of the choice of coherent filtration.
For $U\subset X$ and a coherent filtration $F$ on $\mathcal{M}|_{U}$ one has
$$\mathrm{Ch}(\mathcal{M})\cap \pi_X^{-1}(U)=\mathrm{Supp}\big(\mathcal{O}_{T^*U}\otimes_{\pi_X^{-1}Gr^F\mathcal{D}_U}\pi_X^{-1}Gr^F(\mathcal{M}|_U)\big).$$

For open subsets $\Omega\subset T^*X,$ denote for each $p\in\mathbb{Z}_+$ the space of $\mathcal{O}_{T^*X}(p)$-valued sections on $\Omega$ by $\Gamma\big(\Omega,\mathcal{O}_{T^*X}(p)\big),$ and write 
$$\Gamma\big(\Omega,\mathcal{O}_{T^*X}(\star)\big):=\oplus_{p\in\mathbb{Z}}\Gamma\big(\Omega,\mathcal{O}_{T^*X}(p)\big),\hspace{2mm} as \hspace{1mm} \bigoplus
_p\mathcal{O}_{\Omega}[\xi](p).$$
In particular an element is a sequence $F=\{F_p\}$ with $F_p\in \mathcal{O}_{\Omega}[\xi](p)$ homogeneous of degree $p$, written locally around a point $(x_0,\xi_0)$ by $F_p(x_0,\xi_0)=\sum_{|\sigma|=p}f_p^{\sigma}(x_0)\xi_0^{\sigma}$.

We denote by $\mathcal{E}_X$ the sheaf of micro-differential operators on $X.$

A brief recollection from \cite{KSh}, rephrased as in \cite{CFK,CFK2} is now given.
For a closed $\D$-subscheme $\EQ$, we put $\mathcal{C}h^{\EQ}=\oplus_i \mathcal{N}\otimes\mathcal{O}_{T^*X}(i)$, where for $a>0$ the space $\mathcal{C}h_{\geq a}^{\EQ}$ consists of elements of degree at least $a.$ For all $a\leq b,$ we had set $\mathcal{C}h_{[a,b]}^{\EQ}:=\mathcal{C}h_{\geq a}^{\EQ}/\mathcal{C}h_{\geq b}^{\EQ}$.

The $\D$-Hilbert scheme of $\D$-ideal sheaves is related to a scheme of graded ideals in a graded commutative algebra $A,$ studied in \emph{loc.cit.}
We take $A$ to be $\mathcal{O}_{T^*(X,Z)}$ or $\mathbb{P}\mathcal{O}_{T^*(X,Z)}:=\mathcal{O}_Z\otimes_{\mathcal{O}_X}\mathcal{O}_{\mathbb{P}T^*X},$ as needed.

The scheme of ideals is constructed as a graded $\mathcal{O}_{T^*(X,Z)}$-Grassmannian. It is related to the $\D$-Hilbert scheme and is in fact isomorphic precisely due to involutivity - the $\D$-Hilbert polynomials carry the same paramaterizing datum.
Fixing a numerical polynomial $P$ we are considering sub-sheaves of characteristic $\D$-modules with $\D$-Hilbert polynomial $P$:
$$Sub^{P}(\EQ):=\big\{\mathcal{N}_0\subseteq \mathcal{C}h^{\EQ}|P_{\mathcal{N}_0}=P\big\}.$$
We then set
$$Gr_{\mathcal{D},\mathcal{O}_{T^*(X,\EQ)}}\big(P,\mathcal{C}h\big):=\prod_{i}Gr_{\mathcal{D}}\big(P(i),\mathcal{C}h_{i}\big).$$

It is equivalently described via $\mathcal{A}[\mathcal{E}_X]$-submodules with fixed $\D$-Hilbert polynomial. Following the conventions of \cite{CFK2}, to a (graded) associative algebra $A$ and a left $A$-module $M,$ one has
$Gr_{A}(k,M)\subset Gr(k,M),$
whose scheme of left ideals is $J(k,A):=Gr_{A}(k,A),$ which for our purposes will apply to commutative algebras $A$, thus describes the scheme of two-sided ideals of $A$ whose quotients $A/I$ are again commutative. 

In this case, we obtain
\begin{equation}
    \label{eqn: Microlocal Ideal Scheme}
J_{\mathcal{D}}(k,\mathcal{O}_{T^*(X,Z)}):=Gr_{\mathcal{D},\mathcal{O}_{T^*(X,Z)}}(k,\mathcal{O}_{T^*(X,Z)}).
\end{equation}
A point $[\mathcal{J}]\in J_{\mathcal{D}}(k,\mathcal{O}_{T^*(X,Z)})$ corresponds to a finitely (co)generated graded $\mathcal{O}_{T^*(X,Z)}$-submodule 
$\mathcal{J}\hookrightarrow \mathcal{A}^{\ell}\otimes_{\mathcal{O}_X}\mathcal{O}_{T^*X}.$ 

\begin{prop}
    \label{prop: D-Hilb and Microlocal Ideal Scheme}
    Fix a $\D$-smooth commutative $\D$-algebra of jets $\mathcal{A}$. Then, for $a\leq b$ there exists an isomorphism of schemes,
$$\beta_{[a,b]}:Hilb^P(\mathrm{Spec}_{\mathcal{D}}(\mathcal{A})\times T^*X)\xrightarrow{\simeq} J_{\mathcal{D}}(P,\mathcal{O}_{T^*(X,\mathcal{A})}[a,b]).$$
    Moreover, the functors
    $\underline{\mathcal{H}ilb}_{\mathcal{D}}^{P}(\mathrm{Spec}_{\mathcal{D}}(\mathcal{A})\big)$ and $\underline{Gr}_{\mathcal{D},\mathcal{O}_{T^*(X,\mathcal{A})}}(P,\mathcal{O}_{T^*(X,\mathcal{A})}),$
   are equivalent when restricted to
    $\mathrm{CAlg}_{\mathcal{D}_X}^{inv}\subset \mathrm{CAlg}_X(\mathcal{D}_X).$
\end{prop}
Since any $\mathcal{O}_{T^*(X,\mathcal{A}^{\ell})}$-graded submodule is automatically a graded $\mathcal{O}(T^*(X,\mathcal{A}))_{[a,b]}$-submodule, we have a canonical embedding 
$$\varphi_{[a,b]}:Gr_{\mathcal{D},\mathcal{O}_{T^*(X,Z)}}(k;\mathcal{O}(T^*(X,Z))_{[a,b]})\hookrightarrow J_{\mathcal{D}}(k,\mathcal{O}(T^*(X,Z))_{[a,b]}).$$
These results will be used to realize a single underlying classical finite-type $\D$-scheme as the classical truncation of a dg-$\D$-manifold obtained as a simplicial diagram of dg-schemes modulo actions of Lie pseudogroups. The latter is exactly the $\D$-Quot dg-manifold. This paper is devoted to its construction.

\section{Derived algebraic $\D$-geometry}
\label{sec: Derived D Geometry}
We recall the notions of hypersheaves valued in simplicial sets with respect
to a natural $\D$-étale topology on derived $\D$-algebras (see Subsect. \ref{sssec: D-étale}).

\subsection{Background on derived $\D$-geometry}
\label{ssec: Background on derived D-geometry}

For any complex $\mathcal{M}^{\bullet}$ with a $\mathcal{D}_X$-linear differential, denote the cohomology spaces $\mathcal{H}_{\mathcal{D}}^*(\mathcal{M}).$
In particular, for an affine derived $\D$-scheme $\mathbf{Spec}_{\mathcal{D}}(\mathcal{A}^{\bullet})$ one has that $\mathcal{H}_{\mathcal{D}}^0(\mathcal{A})$ is a sheaf of classical commutative $\D$-algebras on $X$ and $\mathcal{H}_{\mathcal{D}}^{-i}(\mathcal{A}^{\bullet})\in \mathbf{Mod}(\mathcal{H}_{\mathcal{D}}^0(\mathcal{A})\otimes_{\mathcal{O}_X}\mathcal{D}_X),i\geq 0.$ 

\begin{defn}
\normalfont
A \emph{dg-$\D$-algebra on $X$} is a sheaf of differential graded commutative $\D_X$-algebras, $\mathcal{A}^{\bullet}:=(\mathcal{A}^{\sharp},d_{\mathcal{A}}),$ whose cohomologies with respect to the internal cohomological differential $d_{\mathcal{A}}$, are denoted by $\mathcal{H}_{\D}^{i}(\A^{\bullet}):=H^i(\mathcal{A}^{\bullet},d_{\mathcal{A}})$. An element $a\in\mathcal{A}^{\bullet}$ is said to be of homogeneous degree $i\in\mathbb{Z}$ if $a_i\in \mathcal{A}^{i}\subset \mathcal{A}^{\sharp}.$ 
A \emph{derived $\D$-algebra on $X$} is a dg-$\D$-algebra $\mathcal{A}^{\bullet}$ such that $\mathcal{H}_{\D}^0(\mathcal{A}^{\bullet})$ is an ordinary commutative $\D_X$-algebra on $X$, and such that $\mathcal{H}_{\D}^i(\mathcal{A}^{\bullet})$ are finitely presented (left) $\mathcal{H}_{\D}^0(\mathcal{A}^{\bullet})\otimes_{\mathcal{O}_X}\D_X$-modules, for each $i\in\mathbb{Z}.$ 
\end{defn}
We denote by $\mathbf{CAlg}_X(\D_X)$ the dg-category of dg-$\D$-algebras on $X$ with opposite category denoted by $\mathbf{dAff}_X(\D_X).$ An object $Z$ of $\mathbf{dAff}_X(\D_X)$ will be denoted by $Z=\mathbf{Spec}_{\D_X}(\A^{\bullet}),$ and is called a \emph{derived affine $\D$-schemes (over $X$).}

The graded Leibniz rule and $\D$-linearity of the internal differential  $d_{\mathcal{A}}:\mathcal{A}^{\bullet}\rightarrow \mathcal{A}^{\bullet+1}$ are expressed on homogeneous elements as
\begin{eqnarray}
d_{\mathcal{A}}(a_i\cdot a_j)&=&(d_{\mathcal{A}}a_i)\cdot a_j+(-1)^{i}a_i\cdot (d_{\mathcal{A}}a_j),\hspace{2mm} a_i,a_j\in \mathcal{A}^{\bullet},
\\
d_{\mathcal{A}}\big(\mathcal{D}_{X}\bullet a\big)&=&\mathcal{D}_X\bullet (d_{\mathcal{A}}a), \hspace{1mm} a\in \mathcal{A}^{\bullet}.
\end{eqnarray}
We mainly study derived affine $\D$-schemes which model homotopical thickenings of PDEs viewed as algebraic $\D$-schemes, which possess additional finiteness hypothesis.
\begin{defn}
\normalfont
Let $E$ be a vector bundle on $X$ with $\D$-scheme of algebraic infinite jets $J_X^{\infty}E$, and consider a $\D$-algebra $\mathcal{B}$ defined by a differentially generated $\D_X$-ideal $\mathcal{I}$ i.e. $\mathcal{B}=\mathcal{O}(J_X^{\infty}E)/\mathcal{I}.$ Then a derived $\D$-algebra $\mathcal{B}^{\bullet}$ is said to be a \emph{derived enhancement of $\mathcal{B}$} if its underlying classical $\D$-algebra is isomorphic as a $\D$-algebra to a PDE defined by the ideal $\mathcal{I}$ i.e. there exists a surjective morphism of $\D_X$-algebras 
$\mathcal{O}(J_X^{\infty}E)\rightarrow \mathcal{H}_{\D}^0(\mathcal{B}^{\bullet}),$ and an isomorphism $\mathcal{H}_{\D}^0(\mathcal{B}^{\bullet})\simeq \mathcal{B}.$

\end{defn}
Another class of derived $\D$-enhancements we consider are those for which $\mathcal{H}_{\D}^i(\mathcal{B}^{\bullet})$ are induced vector $\D$-bundles over $J_X^{\infty}E$ e.g. isomorphic to $p^*F\otimes \mathcal{O}(J_X^{\infty}E)[\D],$ for some coherent sheaf or vector bundle $F$ on $X$, pulled back along the structure map $p:J_X^{\infty}E\rightarrow X.$
A $\D$-ideal $\I$ naturally defines a derived $\D$-prestack of solutions by cofibrantly replacing the quotient $\mathcal{B}=\mathcal{O}(J_X^{\infty}E)/\mathcal{I}$ in the category of derived $\D$-algebras under $\mathcal{A}:=\mathcal{O}(J_X^{\infty}E)$, viewed as complex concentrated in degree zero.
Such a replacement is of the form
\begin{equation}
    \label{eqn: QRes}
\mathcal{A}\hookrightarrow \mathcal{A}\otimes\mathbf{Free}(\mathcal{M})\xrightarrow{q}\mathcal{B},
\end{equation}
whose initial morphism $\iota$ is an inclusion given by mapping $a\mapsto a\otimes 1$ and $\mathcal{M}$ is a free graded $\mathcal{D}$-module that is compact and $\mathbf{Free}(\mathcal{M})$ is the free derived $\D$-algebra generated by it. The homotopical functor of points gives a pre-stack of solutions,
\begin{equation}
\label{eqn: DPreStk}
\mathbb{R}\mathbf{Sol}_{\mathcal{D}_X}:\mathrm{Ho}(\mathbf{CAlg}_X(\mathcal{D}_X)_{\mathcal{A}}/)\rightarrow \mathrm{Ho}(\mathbf{SSets}),
\end{equation}
also to be denoted by $\mathbb{R}\mathbf{Spec}_{\mathcal{D}}(\mathcal{B}).$

\subsubsection{Cotangent complexes and characteristics}
 For a dg-$\mathcal{D}$-algebra $\mathcal{A}^{\bullet}$, suppose that its cotangent complex $\mathbb{L}_{\mathcal{A}}$ exists and is perfect as an $\mathcal{A}^{\bullet}\otimes \mathcal{D}_X$-module with microlocalization $\mu\mathbb{L}_{\mathcal{A}}$ (see \cite[7.3]{KSY}). For example, let $\mathcal{I}\subset \mathcal{A}$ be a $\D$-PDE with quotient $\D$-algebra $\mathcal{B}$. Then, 
$$\mathbb{L}_{\mathcal{B}}^{\bullet}\simeq \big[\mathcal{I}/\mathcal{I}^2\rightarrow \Omega_{\mathcal{A}/\mathcal{O}_X}^1\otimes_{\mathcal{A}}\mathcal{B}\big],$$
is a complex of $\mathcal{A}\otimes\D_X$-modules
concentrated in degrees $-1,0.$

Considering the pull-back along the cotangent projection,
\begin{equation}
    \label{eqn: DerivedChar}
\mathrm{Char}_{\mathcal{D}}(\mathcal{A}^{\bullet}):=\mathrm{supp}(\mu \mathbb{L}_{\mathcal{A}}) \subset \mathbf{Spec}_{\mathcal{D}}(\mathcal{A})\times_X T^*X.
\end{equation}
We refer to (\ref{eqn: DerivedChar}) as the characteristic variety of the derived $\D$-algebra. With the appropriate finiteness hypothesis, we have a derived bi-duality via the local (that is, $\mathcal{A}[\mathcal{D}]$) Verdier dual $\mathbb{D}^{loc}:=\mathbb{R}\mathcal{H}om_{\mathcal{A}[\mathcal{D}}(-,\mathcal{A}[\mathcal{D}]),$ and one may equivalently work with the characteristic variety in terms of left $\mathcal{A}[\mathcal{D}]$-dual e.g. the twisted dual $\mathbb{T}_{\mathcal{A}}\otimes \omega_X^{\otimes -1},$ (explained further below).

Let $\pi:T^*X\rightarrow X$ be the natural projection and let $\mathcal{E}_X$ be the sheaf of microdifferential
operators on $T^*X.$
It is convenient to introduce the functor 
\begin{equation}
    \label{eqn: MicroLocalVerdier}
    \mathbb{D}^{\mu loc}:=\mathbb{R}\mathcal{H}om_{\pi^*\mathcal{A}[\mathcal{E}]}(-,\pi^*\mathcal{A}[\mathcal{E}]),
    \end{equation}
and call it the `micro' local Verdier dual i.e. $\mathcal{A}[\mathcal{E}]$-dual.

Finally, if $\varphi:\mathcal{A}\rightarrow \mathcal{O}$ is a solution, given by a morphism of $\D$-algebras, the pull-back $\mathbb{T}_{\mathcal{A},\varphi}^{\ell}$ along $\varphi:X\rightarrow \mathbf{Spec}(\mathcal{A})$ gives a perfect complex of $\D$-modules on $X$. Geometrically, it represents the linearization complex of the non-linear system defined by $\mathcal{A}^{\bullet}.$ 

Pull back $\mu(\mathbb{T}_{\mathcal{A}}^{\ell})$ gives the microlocalization of its linearization. Thus, for each $\varphi$ one obtains the $\D$-module characteristic variety for the derived linearization, to be denoted by $\mathrm{Char}(\mathcal{A}^{\bullet},\varphi):=\mathrm{supp}(\mu \mathbb{T}_{\mathcal{A},\varphi}^{\ell})$.

\subsection{Finiteness properties}
\label{ssec: Finiteness}
Homotopical finiteness via (eventual) co-connectivity is necessary for some constructions and required for most proofs in the $\D$-setting. It is most suitably formulated using the language of ind-coherent sheaves on de Rham stacks $X_{DR}$, following \cite{GR,GR2}. In particular, `admitting a cotangent complex,' means admitting a \emph{pro-cotangent complex}.

Let $S$ be an scheme almost of finite type. We use the following result in several places (e.g. when $S$ is of the form $T\times X_{DR}$, for aft-scheme $T$) so give details of its proof.
\begin{prop}
There exists a symmetric monoidal fully-faithful embedding $\mathbb{D}_S:\mathbf{IndCoh}(S)^{op}\rightarrow \mathbf{Pro}\big(\mathbf{QCoh}(S)^-\big).$
\end{prop}
\begin{proof}
    There is functor $\gamma_S:\mathbf{Pro}\big(\mathrm{Coh}(S)\big)\rightarrow \mathbf{Pro}\big(\mathbf{QCoh}(S)^-\big)$ defined by first viewing an object of the domain as an exact functor $\mathrm{Coh}(S)\rightarrow \mathrm{Vect},$ and sending it to its left-Kan extension along $\mathrm{Coh}(S)\hookrightarrow \mathbf{QCoh}(S)^b$, then right Kan extension along $\mathbf{QCoh}(S)^b\rightarrow \mathbf{QCoh}(S)^-.$ Serre duality gives an auto-equivalence
    $\mathbb{D}_S:\mathbf{IndCoh}(S)^{op}\simeq \mathbf{Pro}(\mathrm{Coh}(S)),$ and composing with $\gamma_S$ gives a fully-faithful functor, denoted the same
    $$\mathbb{D}_S:\mathbf{IndCoh}(S)^{op}\hookrightarrow \mathbf{Pro}\big(\mathbf{QCoh}(S)^-\big).$$
    Endow ind-coh with its usual $\otimes^!$-tensor structure, the induce symmetric monoidal structure on the pro-category is induced by Serre-duality equivalence. It is exact in each variable and preserves limits, but is not continuous. There is an induced unique symmetric monoidal on $\mathbf{Pro}(\mathbf{QCoh}(S)^-)$ preserving limits and restricting to the usual one on $\mathbf{QCoh}(S)^-.$
\end{proof}
We always assume the following finiteness condition.

\begin{defn}{\cite[Definition 6.30]{KSY}}
\normalfont
A $\D$-subscheme $Z\subset q_*(E)$ is \emph{$\D$-finitary} if its solution space $\RS_X(Z)$ is locally almost of finite-type and for each solution i.e. $u_T:T\rightarrow \RS_X(Z)$, pull-back defines $u_T^*\mathbb{L}_{Z/X_{DR}}\in \mathbf{Perf}(T\times X_{DR}),$ for every test scheme $T.$
\end{defn}
In \emph{loc.cit.} we worked with $\mathbf{Pro}\big(\mathbf{QCoh}(T\times X_{DR})^-),$ but restrict to perfect objects for simplicity.
Duality in this context is local Verdier duality
\begin{equation}
    \label{eqn: LocVerd}
\mathbb{D}^{loc}:\mathbf{Perf}(T\times X_{DR})\simeq \mathbf{IndCoh}(T\times X_{DR})^{op},
\end{equation}
and the image of $u_T^*\mathbb{L}_{Z/X_{DR}}$ defines the \emph{tangent $\D$-complex of $Z$}, denoted 
\begin{equation}
\label{eqn: Tangent D}
\mathbb{T}_{Z/X_{DR},u_T}\simeq \mathbb{D}^{loc}(u_T^*\mathbb{L}_{Z/X_{DR}}).
\end{equation}
\begin{rmk}
Functor (\ref{eqn: LocVerd}) takes values in Ind-coh. Thus, (\ref{eqn: Tangent D}) is a `right' $\D$-module. Its `left' version may be restored using the functor $\Upsilon_{X_{DR}}^{-1},$ as in \cite{GR}. 
\end{rmk}
For example, by \cite[Prop. 3.0.1]{KSY2}, together with the necessary hypothesis of being $\D$-finitary, for $Z\rightarrow X_{DR}$, there exists a canonical map $\eta:q^*(Z)\rightarrow X,$ with dg-structure sheaf $\mathcal{O}\big(q^*Z\big)$, whose push-forward is a sheaf of dg-$\mathcal{O}_X$-modules, $\eta_*\mathcal{O}\big(q^*Z\big)$. 
Pulling back the relative cotangent complex under $q^*,$ gives (by \cite[Proposition 3.10]{KSY2}), an object of the $\infty$-category,
 $$\eta_*\mathcal{O}\big(q^*Z\big)-\mathrm{Mod}(\D_X)\simeq \mathrm{Mod}(\mathcal{A}_X\otimes\D_X),$$
where $\mathcal{A}_X:=\eta_*\mathcal{O}(q^*Z).$ Its dual, under (\ref{eqn: LocVerd}) is 
\begin{equation}
    \label{eqn: Derived D Tangent}
Maps\big(q^*\mathbb{L}_{Z/X_{DR}},\eta_*\mathcal{O}\big(q^*Z\big)\otimes_{\mathcal{O}_X}\D_X\big)\otimes \omega_X^{\otimes -1},
\end{equation}
with $\omega_X^{\otimes -1},$ the usual dual (see \cite{HT}).

\begin{defn}
\normalfont
    A derived $\D$-algebra is said to be \emph{$n$-coconnective}, if it is concentrated in cohomological degrees $\geq -n.$ Denote the category by
    $\mathbf{CAlg}_X(\mathcal{D}_X)^{\geq -n}.$ Similarly, $n$-connective $\D$-affine scheme (resp. $\D$-prestacks) are 
    objects of the opposite category $\mathbf{DAff}^{\leq n}$ (resp. category of functors $\mathbf{DAff}^{\leq n}\rightarrow \mathbf{Spc}$). 
\end{defn}

Let $U_{\mathrm{alg}}:\mathbf{CAlg}_X(\mathcal{D}_X)\rightarrow \mathbf{Mod}(\mathcal{D}_X)$ denote the forgetful functor, with similar notation on $\mathbf{CAlg}(\mathcal{D}_X)^{\leq -n}.$
Put $\mathbf{Mod}(\mathcal{O}_X):=\mathbf{QCoh}(X)^{\leq 0},$ and `forget the connection' by $\mathrm{For}_{\mathcal{D}}:\mathbf{Mod}(\mathcal{D}_X)\rightarrow \mathbf{Mod}(\mathcal{O}_X).$

Consider the composite forgetful functor on commutative monoid objects
\begin{equation}
    \label{eqn: Forget D-Alg}
For_{\D}:\mathbf{CAlg}_X(\mathcal{D}_X)\rightarrow \mathbf{CAlg}_X(\mathcal{O}_X):=\mathbf{Alg}_{\mathcal{C}\mathrm{omm}}[\mathbf{Mod}(\mathcal{O}_X)].
\end{equation}
The homotopy-coherent algebraic jet functor corresponds to its left adjoint, $(For_{\D})_{Comm}^L.$
The $\infty$-category of $n$-coconnective derived $\D$-algebras is identified with
\begin{equation}
    \label{eqn: n-coconn}
\begin{tikzcd}
\mathbf{CAlg}_X(\mathcal{D}_X)^{\geq -n}\arrow[d]\arrow[r] & \big(\mathbf{QCoh}(X)^{\leq 0}\cap \mathbf{QCoh}(X)^{\geq -n}\big)\arrow[d]
\\
\mathbf{Alg}_{\mathrm{Comm}}[\mathbf{QCoh}(X_{dR})]\arrow[r] & \mathbf{QCoh}(X).
\end{tikzcd}
\end{equation}

Given a morphism $\mathcal{A}^{\bullet}\rightarrow \mathcal{B}^{\bullet}$ in $\mathbf{CAlg}_X(\mathcal{D}_X)^{\leq -n},$ one says that $\mathcal{B}^{\bullet}$ is $\mathcal{D}_X$ \emph{finitely presented} over $\mathcal{A}$ (i.e. as an $\mathcal{A}^{\bullet}-\mathcal{D}_X$-algebra via $f$) if it is a compact object in 
\begin{equation}
    \label{eqn: CAlgSlice}
\mathbf{CAlg}\big(\mathbf{Mod}_{\mathcal{D}_X}(\mathcal{A})^{\leq -n}\big)\simeq \mathbf{CAlg}_X(\mathcal{D}_X)_{\mathcal{A}/}^{\leq -n},
\end{equation}
for some $n.$

Let us denote by $\mathbf{Perf}_{\mathcal{D}_X}(\mathcal{A})$ the sub $\infty$-category of $\mathbf{Mod}_{\mathcal{D}_X}(\mathcal{A})\simeq \mathbf{Mod}(\mathcal{A}\otimes\mathcal{D}_X)$ consisting of perfect complexes of $\mathcal{A}[\mathcal{D}]$-modules.
\begin{defn}
    \label{defn: D-AFP}
    \normalfont 
A derived $\D$-algebra $\mathcal{B}$ is $\mathcal{D}_X$-\emph{almost finitely presented} (over $\mathcal{A})$ if there is a morphism $f:\mathcal{A}\rightarrow \mathcal{B}$ of derived $\D$-algebras such that
$\tau^{\leq -n}(\mathcal{B})$ is a finitely presented $\tau^{\leq -n}(\mathcal{A})$-algebra object in commutative $\D$-algebras via the induced map $\tau^{\leq -n}(f):\tau^{\leq -n}(\mathcal{A})\rightarrow \tau^{\leq -n}(\mathcal{B}),$ for each $n\in \mathbb{Z}.$ 

\end{defn}
In particular, a derived $\D$-algebra $\mathcal{A}$ is $\mathcal{D}_X$-almost finitely presented if it so for the identity morphism i.e. for every $n$, the object $\tau^{\leq -n}(\mathcal{A})$ is compact in $\mathbf{CAlg}_X(\mathcal{D}_X)^{\leq -n}.$

\begin{obs}
\normalfont 
When $X$ is a smooth (quasi)-projective variety, these objects compactly generate the category under filtered colimits. Such compact derived $\D$-algebras are retracts of derived $\D$-algebras $\mathcal{A}\simeq \mathrm{hocolim}_n \mathcal{A}_n,$ with each $\mathcal{A}_n$ free on a compact $\D$-module $\mathcal{M}_n.$ In the case $\mathcal{M}_n$ is further locally free it is equivalent to an object of the form $\mathrm{ind}_{\mathcal{D}}(E_n)$ for a vector bundle. In this case, such objects have well-defined ranks.
\end{obs}

\subsubsection{Postnikov towers and Koszul duals}
We will need two additional technical ingredients that will be used later (e.g. in Subsect. \ref{sssec: Injectivity locus and the quotient}- \ref{sssec: Computing a colimit}). 


To each $\D_X$-space with dg-structure sheaf $\mathcal{O}_{Z}^{\bullet}$ associate its $\D$-geometric Postnikov sequence, denoted by the sequence $\{\mathsf{P}_{\leq n}(\mathcal{O}_{Z})\}_{n\geq 0},$ of objects in derived $\D$-algebras under $\mathcal{O}_Z,$ with a sequence
\begin{equation}
\label{eqn: Postnikov}
\mathcal{O}_{Z}^{\bullet}\rightarrow\cdots\rightarrow \mathsf{P}_{\leq n}(\mathcal{O}_{Z}^{\bullet})\rightarrow \mathsf{P}_{\leq n-1}(\mathcal{O}_Z^{\bullet})\rightarrow \cdots\rightarrow \mathsf{P}^0(\mathcal{O}_{Z}^{\bullet})\simeq \mathcal{H}_{\mathcal{D}}^0(\mathcal{O}_Z),
\end{equation}
which satisfies the properties: \begin{enumerate}
    \item For all $i\geq n,$ one has $\mathcal{H}_{\mathcal{D}}^{-i}\big(\mathsf{P}_{\leq n}\mathcal{O}_Z^{\bullet}\big)\simeq 0;$
    \item For $i\leq n$ the natural map $i_{n}:\mathcal{O}_Z\rightarrow \mathsf{P}_{\leq n}\mathcal{O}_Z,$ induces isomorphisms $\mathcal{H}_{\mathcal{D}}^{-i}(\mathcal{O}_Z)\simeq \mathcal{H}_{\mathcal{D}}^{-i}(\mathsf{P}_{\leq n}\mathcal{O}_Z).$
\end{enumerate}

If $Z$ is finitely $\D$-presented $\D$-space, $\mathcal{H}^0(\mathbb{T}_{Z}^{\ell})$ is a vector $\D$-bundle $\Theta_{Z}$ and its local Chevalley-Eilenberg dg-algebra (relative de Rham complex of the $\D$-scheme $Z$ over $X$) is
\begin{equation}
    \label{eqn: CE}
CE(\Theta_Z):=\mathrm{Sym}^{\otimes^!}(\mathbb{D}^{loc}(\Theta_Z)[-1]).
\end{equation}
This makes sense for any Lie algebroid object in (derived) $\D$-spaces with the appropriate finiteness conditions. Thus, given any such dg Lie algebroid $\mathcal{L}$ over a $\D$-space which is perfect and dualizable in the above sense we denote by $\mathcal{Y}_{\mathcal{L}}$ the derived formal stack induced from (\ref{eqn: CE}). Namely, following \cite{H}, it is obtained via a type of (Koszul) duality \cite{FG}, as the spectrum of a certain cocommutative coalgebra in $\D$-modules i.e.
$\mathcal{O}(\mathcal{Y}_{\mathcal{L}})\simeq \mathrm{CE}(\mathbb{D}^{loc}(\mathcal{L})[-1]).$



\subsubsection{Covering families}
\label{sssec: D-étale}
Fix a $\D$-prestack $Z$ with structure sheaf $\mathcal{O}_Z$ and a finite family of (cofibrant, for simplicity) derived $\D$-algebras $\{\mathcal{A}_{\alpha}^{\bullet}\}_{\alpha\in A}$ for some indexing set $A.$ 
A family of morphisms  
$\{f_{\alpha}:\underline{\mathbf{Spec}}_{\mathcal{D}}(\mathcal{A}_{\alpha})\rightarrow Z\}$ is a \emph{$\D$-étale covering of $Z$} if:
\begin{enumerate}
    \item for each $\alpha\in A,$ we have a fiber sequence 
$$\mathbb{L}_{\mathcal{O}_Z}\otimes_{\mathcal{O}_Z}^{\mathbb{L}}\mathcal{A}_i\rightarrow \mathbb{L}_{\mathcal{A}_i}\rightarrow \mathbb{L}_{\mathcal{A}_i/ \mathcal{O}_Z},$$
and subsequent isomorphisms 
$\mathbb{L}_{\mathcal{O}_Z}\otimes_{\mathcal{O}_Z}^{\mathbb{L}}\mathcal{A}_i\xrightarrow{\sim}\mathbb{L}_{\mathcal{A}_i},$
in $Ho\big(\mathbf{Mod}_{\mathcal{D}}(\mathcal{A}_i)\big)$;

\item 
there exists some finite subset $J\subset A$ such that $\{f_j: j\in J\}$ determine conservative functors
$\mathcal{A}_i\otimes_{\mathcal{A}}^{\mathrm{L}}-:Ho\big(\DG(\mathcal{A})\big)\rightarrow Ho\big(\DG(\mathcal{A}_i)\big).$
\end{enumerate}
Note that (2) is more explicitly stated: for every $n<0,$ we have isomorphism $$\mathcal{H}_{\mathcal{D}}^n(\mathcal{A})\otimes \mathcal{H}_{\mathcal{D}}^0(\mathcal{A}_i)\rightarrow \mathcal{H}_{\mathcal{D}}^n(\mathcal{A}_i),$$ with the degree zero part $\mathcal{H}_{\mathcal{D}}^0(f_i):\mathcal{H}_{\mathcal{D}}^0(\mathcal{A})\rightarrow \mathcal{H}_{\mathcal{D}}^0(\mathcal{A}_i)$ satisfying the following universal property: for each (discrete, classical) commutative $\D$-algebra $\mathcal{C}$ and any nilpotent $\D$-ideal $\mathcal{J}\subset \mathcal{C}$, for every morphism $\mathrm{Spec}_{\mathcal{D}}(\mathcal{C})\rightarrow Z^{cl}\simeq \mathrm{Spec}_{\mathcal{D}}\big(\mathcal{H}_{\mathcal{D}}^0(\mathcal{O}_Z)\big),$ the map $\mathcal{H}_{\mathcal{D}}^0(f_{\alpha})$ induces a bijection
\begin{equation}
    \label{eqn: Classical D-étale}
\mathrm{Hom}_{\mathcal{D}}\big(\mathrm{Spec}_{\mathcal{D}}(\mathcal{C}_X^{\ell}),\mathrm{Spec}_{\mathcal{D}}\big(\mathcal{H}_{\mathcal{D}}^0(\mathcal{A}_{\alpha})\big)\rightarrow \mathrm{Hom}_{\mathcal{D}}\big(\mathrm{Spec}_{\mathcal{D}}(\mathcal{C}_{X}^{\ell}/\mathcal{J}),\mathrm{Spec}_{\mathcal{D}}\big(\mathcal{H}_{\mathcal{D}}^0(\mathcal{A}_{\alpha})\big).
\end{equation}

Note that (1) is equivalent to the condition that for every $\alpha\in A,$ we have  $\mathbb{L}_{\mathcal{A}_{\alpha}/\mathcal{A}}\simeq 0.$ The infinitesimal lifting property (\ref{eqn: Classical D-étale}) is simply the condition of being an étale morphism of classical commutative $\D$-algebras on $X.$

\begin{prop}
\label{prop: Zariski substacks}
    Suppose $\mathbb{R}Z$ is a derived $\D$-stack of finite $\D$-presentation and admits a perfect cotangent complex. Let $Z$ be its classical truncation and suppose that $U\hookrightarrow Z$ is an open $\D$-subspace. Then there exists a $\D$-Zariski open derived substack $\mathbb{R}U\hookrightarrow \mathbb{R}Z,$ defined for each derived $\D$-algebra $\mathcal{B}$ by 
    $$\mathbb{R}U(\mathcal{B})\simeq \mathbb{R}Z(\mathcal{B})\times_{Z(\mathcal{H}_{\mathcal{D}}^0(\mathcal{B}))}U(\mathcal{H}_{\mathcal{D}}^0(\mathcal{B})).$$
\end{prop}
\begin{proof}
This is the $\D$-analog of \cite[Prop. 2.1]{STV}, that we need. It follows immediately from their result and by definition of the $\D$-Zariski topology.
\end{proof}
\subsubsection{Results about flatness}
This subsection collects some technical results used in the construction of the $\D$-Hilbert and Quot dg-schemes. In particular, we give analogs of \cite[Prop. 2.4.1-2.4.3]{CFK}.

\begin{prop}
\label{prop: Tor}
    Let $\mathcal{A}\in \mathbf{CAlg}(\mathcal{D}_X),$ and consider $\DG$-modules $\mathcal{M}^{\bullet},\mathcal{N}^{\bullet},$ over $\mathcal{A}.$ There is a spectral sequence in the abelian category of $\mathcal{D}_X$-modules which converges as:
    $$\mathcal{E}_{i,j}^2=\mathcal{T}or_{\mathcal{H}_{\mathcal{D}}(\mathcal{A})}^i\big(\mathcal{H}_{\mathcal{D}}(\mathcal{M}^{\bullet}),\mathcal{H}_{\mathcal{D}}(\mathcal{N}^{\bullet})\big)\Rightarrow \mathcal{H}_{\mathcal{D}}^{i+j}(\mathcal{M}^{\bullet}\otimes_{\mathcal{A}}^{\Ld}\mathcal{N}^{\bullet}),$$
    where we take the $j$-th homogeneous component of $\mathcal{T}or^i.$
    Moreover, if either $\mathcal{M}^{\bullet}$ or $\mathcal{N}^{\bullet}$ are flat, then there is an isomorphism 
    $$\mathcal{H}_{\mathcal{D}}^{\bullet}(\mathcal{M})\otimes_{\mathcal{H}_{\mathcal{D}}(\mathcal{A})}\mathcal{H}_{\mathcal{D}}^{\bullet}(\mathcal{N})\simeq \mathcal{H}_{\mathcal{D}}^{\bullet}(\mathcal{M}\otimes_{\mathcal{A}}^{\Ld}\mathcal{N}).$$
\end{prop}

\begin{prop}
    Let $\mathcal{M}_1^{\bullet},\mathcal{M}_2^{\bullet},\mathcal{M}_3^{\bullet}\in \mathcal{A}-\DG,$ and let $\alpha:\mathcal{M}_2^{\bullet}\rightarrow \mathcal{M}_3^{\bullet}$ be a weak equivalence.
    Then the following hold:
    \begin{enumerate}
        \item If $\mathcal{M}_1^{\bullet}$ is cofibrant, then $id_{\mathcal{M}_1}\otimes \alpha:\mathcal{M}_1^{\bullet}\otimes_{\mathcal{A}}\mathcal{M}_2^{\bullet}\rightarrow \mathcal{M}_1^{\bullet}\otimes_{\mathcal{A}}\mathcal{M}_3^{\bullet}$ is a weak equivalence;

        \item If both $\mathcal{M}_2^{\bullet},\mathcal{M}_3^{\bullet}$ are cofibrant $\mathcal{D}_X$-modules, then $id_{\mathcal{M}_1}\otimes \alpha:\mathcal{M}_1^{\bullet}\otimes_{\mathcal{A}}\mathcal{M}_2^{\bullet}\rightarrow \mathcal{M}_1^{\bullet}\otimes_{\mathcal{A}}\mathcal{M}_3^{\bullet}$ is a weak equivalence.
    \end{enumerate}
\end{prop}
\begin{proof}
    On one hand, notice that the functor $\mathcal{C}\otimes_{\mathcal{A}}-:\mathcal{A}-\DG\rightarrow \mathcal{A}\DG$ is a left Quillen functor for any cofibrant $\mathcal{A}$-module $\mathcal{C}.$ Indeed, it is easy to see it respects cofibrations as well as trivial cofibrations.
Using the 2 out of 3 property, and the commutative diagram,
\[
\begin{tikzcd}
Q\mathcal{M}_1\otimes_{\mathcal{A}}\mathcal{M}_2\arrow[d,"\simeq"] \arrow[r,"\simeq"] & Q\mathcal{M}_1\otimes_{\mathcal{A}}\mathcal{M}_3\arrow[r,"\simeq"] & \mathcal{M}_1\otimes_{\mathcal{A}}\mathcal{M}_3
    \\
    \mathcal{M}_1\otimes_{\mathcal{A}}\mathcal{M}_2\arrow[urr,"id_{\mathcal{M}_1}\otimes\alpha"]&&
\end{tikzcd}
\]
we see that (2) holds.
\end{proof}
\begin{prop}
    Consider $\mathcal{M}^{\bullet},\mathcal{N}^{\bullet}\in \DG(\mathcal{A}^{\bullet})$ and suppose that $\mathcal{M}^{\bullet}$ is cofibrant. If $\mathcal{H}_{\mathcal{D}}(\mathcal{M})$ is free as a $\mathcal{H}_{\mathcal{D}}(\mathcal{A})[\mathcal{D}_X]$-module, then the natural map $\mathcal{H}_{\mathcal{D}}(\mathcal{M})\otimes_{\mathcal{H}_{\mathcal{D}}(\mathcal{A})}\mathcal{H}_{\mathcal{D}}(\mathcal{N})\rightarrow \mathcal{H}_{\mathcal{D}}(\mathcal{M}\otimes_{\mathcal{A}}\mathcal{N})$ is an isomorphism in $\DG\big(\mathcal{H}_{\mathcal{D}}(\mathcal{A})\big).$
\end{prop}
Consider a dg-$\D$ algebra $\mathcal{A}$ and $\mathcal{M}^{\bullet} \in \mathcal{A}^{\bullet}-\DG.$ We say that $\mathcal{M}^{\bullet}$ is $\mathcal{D}_X$-flat (c.f. Proposition \ref{prop: Tor}), if $\mathcal{H}_{\mathcal{D}}^0(\mathcal{M})$ is a flat $\mathcal{H}_{\mathcal{D}}^0(\mathcal{A}^{\bullet}[\mathcal{D}_X]$-module and for every $i$, the natural map 
$$\mathcal{H}_{\mathcal{D}}^i(\mathcal{A}^{\bullet})\otimes_{\mathcal{H}_{\mathcal{D}}^0(\mathcal{A})}\mathcal{H}_{\mathcal{D}}^0(\mathcal{M})\xrightarrow{\simeq} \mathcal{H}_{\mathcal{D}}^i(\mathcal{M}),$$ is an isomorphism.

\begin{prop}
\label{prop: Flat preserves n-truncations}
An object $\mathcal{M}^{\bullet}\in \mathcal{A}^{\bullet}-\DG$ is $\mathcal{D}_X$-flat is equivalent to the functor $\mathcal{M}\otimes_{\mathcal{A}}-$ preserving $n$-truncations and discrete objects.
\end{prop}
\begin{proof}
    Preservation of truncations is equivalently phrased as commuting with truncation functors $\tau^{\leq -n}$.
    For $n$-truncated $\mathcal{N}$, there is a cofibre sequence
    $$\tau_{\geq n}(\mathcal{M}\otimes_{\mathcal{A}}\mathcal{N})\rightarrow \mathcal{M}\otimes_{\mathcal{A}}\tau_{\leq n}(\mathcal{N})\rightarrow \tau_{\leq n}(\mathcal{M}\otimes_{\mathcal{A}}\mathcal{N}).$$
Moreover, there is a bi-Cartesian sequence associated to the $n$-th transition map in the $\D$-Postnikov tower of $\mathcal{A}^{\bullet}$:
\[
\begin{tikzcd}
\mathcal{H}_{\mathcal{D}}^n(\mathcal{A})\arrow[d]\arrow[r] & 0 \arrow[d]
\\
\tau_{\leq n}\mathcal{A}\arrow[r] & \tau_{\leq n-1}\mathcal{A},
\end{tikzcd}
\]
to which we apply $\mathcal{M}\otimes_{\mathcal{A}}(-)$, to get a co-Cartesian square with fiber given by the $n$-th suspension of $\mathcal{M}\otimes_{\mathcal{A}}\mathcal{H}_{\mathcal{D}}^n(\mathcal{A})$:
\[
\begin{tikzcd}
\big(\mathcal{M}\otimes_{\mathcal{A}} \mathcal{H}_{\mathcal{D}}^n(\mathcal{A})\big)[n]\arrow[d]\arrow[r] & 0\arrow[d]
\\
\mathcal{M}\otimes_{\mathcal{A}}\tau_{\leq n}(\mathcal{A})\arrow[r] & \mathcal{M}\otimes_{\mathcal{A}}\tau_{\leq n-1}(\mathcal{A}).
\end{tikzcd}
\]
Since $\mathcal{H}_{\mathcal{D}}^0(-)$ is symmetric monoidal as a functor between the abelian hearts of the associated $t$-structures, we get
$$\mathcal{M}\otimes_{\mathcal{A}}\mathcal{H}_{\mathcal{D}}^n(\mathcal{A})\simeq \mathcal{H}_{\mathcal{D}}^0\big(\mathcal{M}\otimes_{\mathcal{A}}\mathcal{H}_{\mathcal{D}}^n(\mathcal{A})\big)\simeq \mathcal{H}_{\mathcal{D}}^0(\mathcal{M})\otimes_{\mathcal{H}_{\mathcal{D}}^0(\mathcal{A})}\mathcal{H}_{\mathcal{D}}^n(\mathcal{A}).$$
Finally, notice that since $\mathcal{M}\otimes_{\mathcal{A}}(-)$ commutes with $n$-th truncations the latter co-Cartesian square becomes the fiber of the $n$-th transition maps of the $\D$-Postnikov tower of $\mathcal{M}$,
\[
\begin{tikzcd}
\big(\mathcal{H}_{\mathcal{D}}^0(\mathcal{M})\otimes_{\mathcal{H}_{\mathcal{D}}^0(\mathcal{A})}\mathcal{H}_{\mathcal{D}}^n(\mathcal{A})\big)[n]\arrow[d]\arrow[r] & 0\arrow[d]
\\
\tau_{\leq n}(\mathcal{M})\arrow[r] & \tau_{\leq n-1}(\mathcal{M}).
\end{tikzcd}
\]
\end{proof}

\begin{prop}
\label{prop: D-flat result}
    A morphism $f:\mathcal{A}\rightarrow \mathcal{B}$ of derived $\D$-algebras is $\D$-flat if and only if $\mathcal{B}\otimes_{\mathcal{A}}^{\mathbb{L}}\mathcal{H}_{\mathcal{D}}^0(\mathcal{A})$ is weakly-equivalent to a discrete flat $\mathcal{H}_{\mathcal{D}}^0(\mathcal{A})$-algebra in $\D_X$-modules.
\end{prop}
\begin{proof}
Follows from Proposition \ref{prop: Tor} and Proposition \ref{prop: Flat preserves n-truncations} via spectral-sequence arguments. That is if $f^*\mathcal{H}_{\mathcal{D}}^0(\mathcal{A})\simeq \mathcal{H}_{\mathcal{D}}^0(\mathcal{B})$ is $\mathcal{H}_{\mathcal{D}}^0(\mathcal{A})$-flat one has 

$$\mathcal{H}_{\mathcal{D}}^p\big(f^*\mathcal{H}_{\mathcal{D}}^q(\mathcal{A})\big)\Rightarrow \mathcal{H}_{\mathcal{D}}^{p+q}(\mathcal{B}),$$ such that $\mathcal{H}_{\mathcal{D}}^j(\mathcal{B})\simeq \mathcal{H}_{\mathcal{D}}^0(\mathcal{B})\otimes_{\mathcal{H}_{\mathcal{D}}^0(\mathcal{A})}\mathcal{H}_{\mathcal{D}}^j(\mathcal{A}),$ for each $j.$
\end{proof}

\begin{cor}
Consider Proposition \ref{prop: D-flat result}. Then the $\D$-algebra $\mathcal{B}$ is homotopically flat over $\mathcal{A}$ if and only if for every $\mathcal{H}_{\mathcal{D}}^0(\mathcal{A})$-algebra $\mathcal{C}$ for every $i>0,$ one has $\mathcal{H}_{\mathcal{D}}^i(\mathcal{B}\otimes_{\mathcal{A}}^{\mathbb{L}}\mathcal{C})\simeq 0.$
\end{cor}

Given a morphism of derived $\mathcal{D}_X$-schemes $\varphi:\mathcal{X}\rightarrow \mathcal{Y},$ notice that any $\mathcal{M}_{\mathcal{X}}^{\bullet} \in \mathcal{A}^{\bullet}-\DG$, quasi-coherent over $X$, is homotopically flat over $\mathcal{Y}$ if for any discrete $\mathcal{O}_{\mathcal{Y}}[\mathcal{D}_X]$-module $\mathcal{N}_{\mathcal{Y}}\simeq \mathcal{H}_{\mathcal{D}}^0(\mathcal{N}_{\mathcal{Y}}),$ we have $\mathcal{M}_{\mathcal{Y}}^{\bullet}\otimes^{\mathbb{L}}\varphi^*\mathcal{N}_{\mathcal{Y}},$ is discrete. 

\begin{prop}
\label{prop: D-epi is local}
The property of being a $\D$-epimorphism (a $\D$-surjection on $\mathcal{H}_{\mathcal{D}}^0$) is a $\D$-étale local property.
\end{prop}
\begin{proof}
Suppose $Z$ is a $\D$-prestack with a $\D$-étale cover $\{f_i:\mathcal{U}_i\rightarrow Z\}.$ Then if $\varphi:\mathcal{M}^{\bullet}\rightarrow \mathcal{N}^{\bullet}$ is a morphism of $\mathcal{O}_{Z}[\mathcal{D}]$-modules, we must show if $f_{i}^*(\varphi)$ is a $\D$-epimorphisms for each $i$, then so is $\varphi.$
Observe that being a $\D$-epimorphism is equivalent to the triviality of the induced map 
$\varphi':\mathcal{H}_{\mathcal{D}}^0(\mathcal{N})\rightarrow \mathcal{H}_{\mathcal{D}}^0(\mathrm{Cone}(\varphi)).$ Since $f_{i}$ is a $\D$-étale map, it is in particular $\D$-flat and thus pull-back is exact and commutes with cohomologies. Moreover, one can verify that $f_i^*\big(\mathcal{H}_{\mathcal{D}}^0\mathrm{Cone}(\varphi)\big)\simeq \mathcal{H}_{\mathcal{D}}^0\big(f_i^*\mathrm{Cone}(\varphi)\big)$ for each $i$ and from the commutation of $f_i^*$ with $\mathcal{H}_{\mathcal{D}}^0,$ we have
$$f_i^*(\varphi'):\mathcal{H}_{\mathcal{D}}^0\big(f_i^*\mathcal{N}\big)\rightarrow \mathcal{H}_{\mathcal{D}}^0\big(\mathrm{Cone}(f_i^*\varphi)\big),$$
which is trivial for all $i,$ since $f_{i}^*(\varphi)$ is a $\D$-epimorphism. Thus $\varphi'\simeq 0.$
\end{proof}

\section{Derived moduli of $\D$-ideal sheaves}
\label{sec: Derived moduli of D-ideals}
Fix a numerical polynomial $P\in \mathbb{Q}[z].$ In this section we construct a derived enhancement of the moduli space of $\D$-involutive (Spencer regular) ideal sheaves $\mathcal{I}$ in a fixed ambient jet $\D$-scheme $J_X^{\infty}E,$ with numerical polynomials $P_{\D}(\mathcal{I},z)=P(z).$ 
Note that if $X$ is complete and $E$ is projective, then $J_X^{\infty}E$ is a scheme. 
\begin{cons}
\label{cons: RQuot and conditions}
\normalfont 
Fix a smooth ambient $\D$-scheme $Z$ with structure sheaf $\mathcal{O}_Z$ over a $\D$-affine proper $k$-scheme $X$ and fix a sheaf of finitely presented $\mathcal{O}_Z[\mathcal{D}_X]$-modules $\mathcal{M}.$ For a derived $\D$-algebra $\mathcal{B}$ set $\mathcal{M}_{\mathcal{B}}:=\mathcal{M}\otimes\mathcal{B}$ to be the perfect complex on $Z_{\mathcal{B}}=Z\times \mathbf{Spec}_{\mathcal{D}}(\mathcal{B}).$
Consider the functor
\begin{equation}
    \label{eqn: DerDQuot}
    \mathbb{R}\underline{Q}_{\mathcal{D}_X,Z}(\mathcal{M}):\mathbf{CAlg}(\mathcal{D}_X)\rightarrow \infty\mathbf{Grpd},
\end{equation}
whose $\mathcal{B}^{\bullet}$-points $\mathbb{R}\underline{Q}_{\mathcal{D}_X,Z}(\mathcal{B}^{\bullet})$ are given by the full-sub $\infty$-groupoid of $\mathbf{Perf}(Z_{\mathcal{B}})_{\mathcal{M}_{\mathcal{B}}}^{\leq 0,\simeq}$ spanned by maps 
$\varphi:\mathcal{M}_{\mathcal{B}}\rightarrow \mathcal{N},$ such that $\mathcal{N}$ is $\mathcal{D}_X$-flat over $\mathbf{Spec}_{\mathcal{D}_X}(\mathcal{B}^{\bullet})$ such that $\varphi$ is a $\D$-epimorphism (surjective on $\mathcal{H}_{\mathcal{D}}^0$) and whose cohomology sheaves have proper supports. 
 In particular, for some dg-$\mathcal{D}$-algebra $\mathcal{C}^{\bullet},$ (over $Z$), as both conditions are stable
under derived base change we indeed have a functor. The conditions state that,
\begin{equation}
    \label{eqn: ProperSuppQuotCondition}
\mathrm{Char}_{\mathcal{D}_X}(\mathcal{C}^{\bullet})\rightarrow \mathbb{R}\mathrm{Spec}_{\mathcal{D}_X}(\mathcal{B}^{\bullet}),
\end{equation}
is a proper morphism of $\mathcal{D}_X$-schemes (relative to $Z$) and that
$\mathcal{H}_{\mathcal{D}}^0(\varphi):H^0(\mathcal{A}^{\bullet}\otimes\mathcal{B}^{\bullet})\rightarrow \mathcal{H}_{\mathcal{D}}^0(\mathcal{C}^{\bullet})$ is a surjective morphisms of classical $\mathcal{D}_X$-algebras satisfying:
\begin{eqnarray}
    \label{eqn: HilbProperty}
    \begin{cases}
        \mathcal{H}_{\mathcal{D}}^0(\mathcal{C}^{\bullet}) \textit{ has Hilbert polynomial } P \textit{ when restricted to } 
        \\
        \textit{ a fiber of } Z\times \mathrm{Spec}_{\mathcal{D}_X}\big(\mathcal{H}_{\mathcal{D}}^0(\mathcal{B}^{\bullet})\big)\rightarrow \mathrm{Spec}_{\mathcal{D}_X}\big(\mathcal{H}_{\mathcal{D}}^0(\mathcal{B}^{\bullet})\big), 
        \\
        \text{ with } \mathcal{H}_{\mathcal{D}}^0(\mathcal{C}^{\bullet}) \text{ is } \mathcal{D}_X-\text{flat over } \mathrm{Spec}_{\mathcal{D}_X}\big(\mathcal{H}_{\mathcal{D}}^0(\mathcal{B}^{\bullet})\big).
        \end{cases}
\end{eqnarray}
In particular, $\mathcal{H}_{\mathcal{D}}^0(\varphi)$ is a surjective morphism of commutative $\mathcal{D}_X$-algebras.
\end{cons}

Condition (\ref{eqn: ProperSuppQuotCondition}) on the relative characteristic variety in the Construction \ref{cons: RQuot and conditions} is the $\mathcal{D}_X$-geometric analog of the condition that its support (microsupport of solution complex intersected with the zero-section $X$) is proper over the classical truncation $\mathrm{Spec}_{\mathcal{D}_X}\big(\mathcal{H}_{\mathcal{D}}^0(\mathcal{B})\big).$ In other words, derived non-linear PDE defined by $\mathcal{C}^{\bullet}$ is properly (micro)supported over $Z\times \mathbb{R}\mathrm{Spec}_{\mathcal{D}_X}(\mathcal{B}^{\bullet})$ relative to $\mathbb{R}\mathrm{Spec}_{\mathcal{D}_X}(\mathcal{B}^{\bullet}).$  

This morphism is obtained as follows. Given a $\mathcal{B}^{\bullet}$-point $\varphi:\mathcal{A}^{\bullet}\otimes\mathcal{B}^{\bullet}\rightarrow \mathcal{C}^{\bullet},$ consider
\begin{equation}
    \label{eqn: Cofibseq}
\mathbb{L}_{\mathcal{A}\otimes\mathcal{B}}\otimes^{\mathbb{L}}\mathcal{C}\rightarrow \mathbb{L}_{\mathcal{C}}\rightarrow \mathbb{L}_{\mathcal{C}/\mathcal{A}\otimes\mathcal{B}}.
\end{equation}
Since $X$ is ($D$-)affine, let $\mathcal{E}_X$ denote the sheaf of microdifferential operators. Apply the functor of twisted microlocalization $_{\mathcal{A}}\mu$ along $\mathcal{A}$ to (\ref{eqn: Cofibseq}) gives,
$$_{\mathcal{A}}\mu\big(\mathbb{L}_{\mathcal{A}\otimes\mathcal{B}}\otimes^{\mathbb{L}}\mathcal{C}\big)\rightarrow _{\mathcal{A}}\mu\big(\mathbb{L}_{\mathcal{C}}\big)\rightarrow _{\mathcal{A}}\mu\big(\mathbb{L}_{\mathcal{C}/\mathcal{A}\otimes\mathcal{B}}\big).$$
Take the (cohomological) supports relates the $\D$-characteristic varieties. In particular, we obtain a map
\begin{equation}
\label{eqn: Char map}
\mathrm{Char}_{\mathcal{D}_X}(\mathcal{C}^{\bullet})=\mathrm{supp}\big(_{\mathcal{A}}\mu\mathbb{L}_{\mathcal{C}}\big)\rightarrow Z\times^h\mathbb{R}\mathrm{Spec}_{\mathcal{D}_X}(\mathcal{B})\rightarrow \mathbb{R}\mathrm{Spec}_{\mathcal{D}_X}(\mathcal{B}).
\end{equation}
Then, (\ref{eqn: Char
map}) is asked to be proper and as $\mathrm{supp}(\mu\mathbb{L}_{\mathcal{C}}^{\bullet})=\cup_{j}\mathrm{supp}(\mathcal{H}_{\mathcal{D}}^j(\mu\mathbb{L}_{\mathcal{C}}^{\bullet})\big)$, each component is proper over $\mathrm{Spec}_{\mathcal{D}}\big(\mathcal{H}_{\mathcal{D}}^0\mathcal{B})\big).$
\vspace{1.5mm}

The functor $\mathbb{R}\underline{Q}_{\mathcal{D}_X,Z}^{P,pre},$ makes sense for sheaves of perfect complexes of $\mathcal{O}_{Z}[\mathcal{D}_X]$-modules for any homotopically finitely presented $\mathcal{D}_X$-scheme $Z,$ and thus parameterizes homotopically $\mathcal{D}_X$-flat families of derived sheaf quotients. 

\begin{notate}
\label{notate: Special notation}
\normalfont 
We reserve the special notation, 
$\mathbb{R}{\mathbf{Quot}}_{\mathcal{D}_X,Z}:=\mathbb{R}\underline{Q}_{\mathcal{D}_X,Z}^{P}(\mathcal{C}h).$
\end{notate}
We make the following definition.

\begin{defn}
\normalfont 
Consider (\ref{eqn: DerDQuot}) and fix a derived $\D$-algebra $\mathcal{B}$. Suppose $P$ is a constant numerical polynomial. The \emph{derived $\D$-Quot functor} of a smooth involutive $\D$-scheme $Z$ is the simplicial $\D$-presheaf $\mathbf{Quot}_{\mathcal{D}_X,Z}^P$ assigning to $\mathcal{B}$ the $\infty$-groupoid of maps: $\mathcal{O}_Z\otimes \mathcal{B}\rightarrow \mathcal{M}$ such that the induced map $\mathcal{O}_Z\otimes \mathcal{H}_{D}^0(\mathcal{B})\rightarrow \mathcal{H}_{\mathcal{D}}^0(\mathcal{M})$ is a $\D$-surjection and $\mathcal{M}$ is a differential graded $(\mathcal{O}_{Z_{\mathcal{B}}})[\mathcal{D}_X]$-module flat over $\mathcal{B}$ such that the $\mathcal{H}_{\mathcal{D}}^0(\mathcal{B})[\mathcal{D}_X]$-module $\mathcal{H}_{\mathcal{D}}^0(\mathcal{M})$ is a vector $\D$-bundle over $Z$ of rank $P.$
\end{defn}

\begin{prop}
\label{prop: DerDQuot is stack}
The simplicial $\D$-presheaf \emph{(\ref{eqn: DerDQuot})} is a derived $\D$-prestack satisfying Zariski-descent \ref{sssec: D-étale}.
\end{prop}
\begin{proof}
Follows from the fact that $\mathbf{Perf}$ is a stack, and that the property of being $\D$-epimorphic and $\D$-flat
are local for the $\D$-étale topology by Proposition \ref{prop: D-epi is local}.
\end{proof}

Condition (\ref{eqn: HilbProperty}) ensures the Hilbert polynomial of the quotient is fixed and from (the proof of) Proposition \ref{prop: Comparison truncation} given below, is the $\mathcal{D}_X$ Hilbert polynomial one would expect in the smooth un-derived situation on the classical truncation. 

We introduce the $\D$-prestack
$$\mathbb{R}\underline{\mathbf{Hilb}}_{\mathcal{D}_X,Z}^{P}:\mathbf{CAlg}(\mathcal{D}_X)\simeq \mathbf{dSch}_X(\mathcal{D}_X)^{op}\rightarrow \infty\mathbf{Grpd}.$$

\begin{defn}
    \label{defn: Derived D Hilbert}
    \normalfont  The \emph{derived} $\mathcal{D}_X$-\emph{Hilbert functor} of the $\mathcal{D}_X$-subscheme $Z$ is the derived (dg)-$\D$ ideal scheme we will always denote by $\mathbb{R}\mathbf{Hilb}_{\mathcal{D}_X,Z}^{P}.$
\end{defn}
Recall that if $\mathrm{dSt}_{\mathbb{C}}$ is the $\infty$-category of derived stacks, it is common to 
denote by $\tau_0:\mathrm{dSt}_{\mathbb{C}}\rightarrow \mathrm{St}_{\mathbb{C}}$, the truncation functor between derived and underived stacks over
$\mathbb{C}$ for the étale topologies (see \cite[Definition 2.2.4.3]{TV}). Similar notation is used for objects of $\mathrm{dStk}_{X}(\D)$.

\begin{prop}
\label{prop: Comparison truncation}
Let $Z$ be $\D$-smooth involutive $\D$-scheme and fix a numerical polynomial $P.$ 
    There is a canonical morphism of derived $\mathcal{D}_X$-prestacks
    $$\psi:\mathrm{Hilb}_{\mathcal{D}_X,Z}^{P}\rightarrow \tau_0\mathbb{R}\mathbf{Hilb}_{\mathcal{D}_X,Z}^{P},$$
    such that its $\mathcal{C}$-points for a reduced discrete $\mathcal{D}_X$-algebra $\mathcal{C}$, gives an isomorphism of connected components e.g.  
$$\pi_0(\tau):\pi_0\big(\mathrm{Hilb}_{\mathcal{D}_X,Z}^{P}(\mathrm{Spec}_{\mathcal{D}_X}(\mathcal{C})\big)\simeq \pi_0\big(\tau_0\mathbb{R}\mathbf{Hilb}_{\mathcal{D}_X,Z}^{P}(\mathrm{Spec}_{\mathcal{D}_X}(\mathcal{C})\big).$$
\end{prop}
\begin{proof}
Firstly, notice that 
$\underline{\mathbf{Quot}}_{\mathcal{D}_Z,Z}^P(Z)(\mathcal{B}^{\bullet})$ is the $\infty$-groupoid of $\varphi:\mathcal{O}_Z\otimes\mathcal{B}\rightarrow\mathcal{C}_1$ with
$$\mathcal{O}_{Z}\otimes\mathcal{H}_{\mathcal{D}}^0(\mathcal{B})\simeq\mathcal{H}_{\mathcal{D}}^0(\mathcal{O}_Z\otimes\mathcal{B})\rightarrow \mathcal{H}_{\mathcal{D}}^0(\mathcal{C}_1),$$
is a $\D$-epimorphism and $\mathcal{C}_1$ is a $\mathcal{O}_{Z_{\mathcal{B}}}[\mathcal{D}_X]$-dg module which is moreover $\D$-flat over $\mathcal{B}$ and such that the $\mathcal{H}_{\mathcal{D}}^0(\mathcal{B})[\mathcal{D}_X]$-module $\mathcal{H}_{\mathcal{D}}^0(\mathcal{C}_1)$ is finitely presented. Then, it is not difficult to observe that 
$$\mathrm{Hilb}_{\mathcal{D}_X,Z}(\mathcal{H}_{\mathcal{D}}^0(\mathcal{B}))\simeq \mathrm{Q}_{\mathcal{D}_X,Z}(\mathcal{H}_{\mathcal{D}}^0(\mathcal{B})).$$
Indeed, the truncation $Q_{\mathcal{D},Z}$ of functor (\ref{eqn: DerDQuot}) is easily viewed as having those $\mathcal{B}$-points given by epimorphisms of $\mathcal{D}_X$-algebras $\mathcal{O}_Z\otimes\mathcal{B}\xrightarrow{\varphi_1}\mathcal{C}_1$ where $\mathcal{C}_1\in QCoh(Z\times \mathrm{Spec}_{\mathcal{D}_X}(\mathcal{B}))\simeq \mathcal{O}_{Z\times \mathrm{Spec}_{\mathcal{D}_X}(\mathcal{B})}-QCoh(\mathcal{D}_X),$ which are moreover finitely $\mathcal{D}_X$-presented, $\mathcal{D}_X$-flat and properly microsupported relative to $\mathrm{Spec}_{\mathcal{D}_X}(\mathcal{B}),$ modulo the equivalence relation that $\varphi_1\simeq \varphi_2$ if there exists a weak equivalence of $\mathcal{D}_X$-algebras (here, an isomorphism) $\gamma:\mathcal{C}_1\rightarrow \mathcal{C}_2$ for which $\gamma\circ \varphi_1=\varphi_2.$
This yields our desired description of the classical prestack $\mathbf{Hilb}_{Z}^{P}:\mathrm{CAlg}_{\mathcal{D}_X}\rightarrow Grpd.$ 
In the standard way, we can view any classical point of this functor as a derived point, by viewing all objects concentrated in cohomological degree zero. 

Conversely,
any two $\mathcal{B}^{\bullet}$-points of $\mathbb{R}Q_{\mathcal{D},Z}$ connected by a weak-equivalence $\gamma$ inducing a homotopy commutative diagram
\[
\begin{tikzcd}
 \mathcal{O}_{Z}\otimes^{\mathbb{L}}\mathcal{B}^{\bullet}\arrow[r,"\varphi_1"] \arrow[d,"\varphi_2"] & \mathcal{C}_1^{\bullet}\arrow[dl,"\gamma"]
    \\
    \mathcal{C}_2^{\bullet} & 
\end{tikzcd}
\]
gives a classical point on the underlying classical truncation of the form
\[
\begin{tikzcd}
\mathcal{H}_{\mathcal{D}}^0\big(\mathcal{O}_{Z}\otimes^{\mathbb{L}}\mathcal{B}^{\bullet}\big)\arrow[r,"H^0(\varphi_1)"] \arrow[d,"H^0(\varphi_2)"] & \mathcal{H}_{\mathcal{D}}^0(\mathcal{C}_1^{\bullet})\arrow[dl]
    \\
    \mathcal{H}_{\mathcal{D}}^0(\mathcal{C}_2^{\bullet}) & 
\end{tikzcd}
\]
where the diagonal morphism $\mathcal{H}_{\mathcal{D}}^0(\gamma)$ is a $\mathcal{D}_X$-algebra isomorphism. 
On the other hand, if $\varphi$ is an object of the connected component of the classical truncation, since $\mathcal{C}^{\bullet}$ is $\mathcal{D}_X$-flat by \cite[Proposition 4.14]{KSh} we have a weak-equivalence of $\mathcal{D}_X$-algebras $\mathcal{C}^{\bullet}\simeq \mathcal{H}_{\mathcal{D}}^0(\mathcal{C})$ and the induced classical point from $\varphi$ is simply 
$\mathcal{H}_{\mathcal{D}}^0(\varphi):\mathcal{H}_{\mathcal{D}}^0(\mathcal{O}_{Z\times \mathrm{Spec}_{\mathcal{D}}(\mathcal{B})})\rightarrow \mathcal{H}_{\mathcal{D}}^0(\mathcal{C}).$
\end{proof}

\subsection{Microlocal stratifications}
In this subsection we give details on \cite[Definition 5.11, Proposition 5.12]{KSh}. Namely, we explain the construction of the $\Lambda$-constrained $\D$-Hilbert functor, 
\begin{equation}
\label{eqn: MicrolocalHilb}
\mathcal{H}ilb_{\mathcal{D}_X}^{P,\Lambda}(Z):=\mathcal{H}ilb_{\mathcal{D}_X}^{P}(Z)\cap \mathcal{H}ilb_{\mathcal{D}_X}^{\Lambda}(Z),
\end{equation}
where $\Lambda\subset T^*X$ and its projection to $X$ is smooth.
The sub-moduli functor (\ref{eqn: MicrolocalHilb}) controls allowable singularities.
To reduce technical machinery on $\infty$-categorical constructions, assume the following:
\begin{ass}
\label{Assumptions}
All $X_{DR}$-spaces are $\D$-finitary and admit perfect cotangent complexes with duals given via (\ref{eqn: LocVerd}).
\end{ass}
With the above assumptions, we write $\mathbf{Perf}(Z)$ in place of $\mathbf{Pro}\big(\mathbf{QCoh}(Z)^-).$

Following \cite[Proposition 5.16]{KSY}, given a morphism $F:Z_1\rightarrow Z_2$ in $\mathbf{PStk}_{X_{DR}}^{\D-fin},$ there exists, for each solution $u_T$, an object $T^*(Z_1/Z_2)_{u_T}\in \mathbf{Perf}(T\times X_{DR})$ by pull-back of $T^*(Z_1/Z_2)\in \mathbf{Perf}(Z_1)$.
The relative (pro)-cotangent complex fits into the cofiber sequence
\begin{equation}
\label{eqn: ProCotangent Cofiber}
F^{\sharp}T^*(Z_2/X_{DR})\rightarrow T^*(Z_1/X_{DR})\rightarrow T^*(Z_1/Z_2),
\end{equation}
in $\mathbf{Perf}(Z_1),$ which we remind is the category of perfect $\mathcal{O}_{Z_1}$-dg-modules in $\mathbf{QCoh}(X_{DR}).$ It is denoted by $\mathbf{Perf}_{\D_X}(Z_1).$

\begin{prop}
Let $X$ be a smooth proper $k$-scheme. Consider $E\rightarrow X$ a prestack locally of finite type over $X$. Then,
$$\mathbf{Perf}(q_*E)\simeq \underset{S\rightarrow E}{\mathrm{lim}}\mathbf{Perf}(\mathcal{A}_S^{\bullet}\otimes \D_X),$$
where $\mathcal{A}_S^{\bullet}:=\mathcal{O}(q_*S)$.
\end{prop}
\begin{proof}
First, consider the affine case. That is, suppose $E\rightarrow X$ affine over $X$, so $E\simeq \mathrm{Spec}_X(R)$ for $R^{\bullet}\in \mathrm{QCAlg}(X).$ Then $\mathcal{A}^{\bullet}:=\mathcal{O}(q_*(E))\simeq (For_{\D})_{comm}^L(R^{\bullet})\in \mathrm{QCAlg}(X_{DR}),$ where $For_{\D}^L$ is the left-adjoint to (\ref{eqn: Forget D-Alg}). Then $\mathbf{Perf}(q_*E)$ is equivalent to 
$\mathbf{Perf}(\mathcal{A}^{\bullet}\otimes \D_X)$. By the left-adjoint property, it commutes with colimits and the assignment of the $\infty$-category $\mathbf{Perf}$ sends colimits to limits.
\end{proof}
In particular, if $R=\mathrm{Sym}(\F)$ for $\F \in \mathbf{QCoh}(X),$ then $\mathcal{A}^{\bullet}\simeq Sym(\D_X\otimes \F)$.

From (\ref{eqn: ProCotangent Cofiber}), it follows the assignment
$\iota:Z\subset q_*(E)\mapsto T^*(Z/q_*E),$ thus in particular $Z\mapsto\mathbb{L}_{Z/X_{DR}}$ is functorial and fournishes an $\infty$-functor
\begin{equation}
    \label{eqn: L assignment}
\mathbb{L}:(Z\subset q_*E)\in\mathrm{PDE}_X(E)^-\longrightarrow \mathbb{L}_{Z/X_{DR}}\in\mathbf{Perf}_{\mathcal{D}_X}(Z),
\end{equation}
where $\mathrm{PDE}_X(E)^{-},$ is the sub-$\infty$-groupoid of $X_{DR}$-prestacks relative to $q_*E,$ which define closed $\D$-subschemes on classical truncations which are moreover $\D$-finitary. See \cite[Observation. 6.2]{KSY}, for details.
Here we promote it to an $\infty$-categorical assignment.
\begin{prop}
\label{prop: Moduli of PDE} 
    Let $X$ be a smooth and proper $k$-scheme with de Rham space $X_{DR}.$ Consider the assignment
    \begin{eqnarray}
        \label{eqn: Derived PDE Moduli}
\mathbb{R}\underline{\mathrm{PDE}}_{X}(-):\mathbf{PStk}_{/X_{dR}}&\rightarrow& \mathbf{PStk}_{/^{\mathrm{cl-emb}}}^{\mathcal{D}-\mathrm{fin},\mathrm{Tor},\simeq}
\\
(Z\rightarrow X_{dR})&\mapsto& \mathbb{R}\underline{\mathrm{PDE}}_{X}(Z) \nonumber
\end{eqnarray}
where $\mathbb{R}\underline{\mathrm{PDE}}_{X}(Z)$ is the maximal sub $\infty$-groupoid consisting of the space of 
closed substacks of $Z$ i.e. generalized non-linear PDEs which are locally of finite Tor-amplitude.
Then, the functor \emph{(\ref{eqn: Derived PDE Moduli})} satisfies both $X$-local $fppf$-descent and proper-surjective descent and respects formally étale morphisms of relative prestacks.
\end{prop}
\begin{proof}
The desired functor is given by composition
$$\mathbf{DStk}_{/X}\xrightarrow{p_{dR*}}\mathbf{DStk}_{/X_{dR}}\rightarrow \big(\mathbf{PStk}_{/^{cl-emb}}^{\mathcal{D}-\mathrm{fin},\mathrm{Tor}}\big)^{\simeq},$$
sending $(E\rightarrow X)$ to $(\mathrm{Jet}_{dR}^{\infty}(E)\rightarrow X_{dR})$ and assigning to this relative pre-stack the space of closed embeddings $i:\mathcal{Y}\hookrightarrow \mathrm{Jet}_{dR}^{\infty}(E)$ which are $\mathcal{D}$-finitary non-linear PDEs. In particular, $\mathbb{L}_{i}$ exists and is dualizable.
Equivalently, the functor factors via
$$\mathbf{DStk}_{/X}\xrightarrow{p_{dR*}}\mathbf{PStk}_{X}(\mathcal{D}_X)^{\mathcal{D}-fin}\rightarrow \big(\mathbf{PStk}_{/^{cl-emb}}^{\mathrm{laft},\mathrm{Tor}}\big)^{\simeq}.$$

Regarding Tor-amplitude $0$ is straightforwardly seen. Namely, for every $\mathcal{D}_X$-scheme $\mathcal{S}$ put
\begin{equation}
    \label{eqn: RPDE-DX}
    \mathrm{PDE}_{\mathcal{D}_X}(Z)^{\mathrm{Tor}_0}(\mathcal{S})\subset \mathbf{dSch}_{\mathcal{D}_X}(Z\times \mathcal{S}),
    \end{equation}
to be the non-full sub category whose objects are derived $\mathcal{D}_X$-schemes,
$i:\mathcal{Y}\rightarrow Z\times \mathcal{S},$
which are homotopically $\mathcal{D}_X$-flat over $\mathcal{S}$ and finitely $\mathcal{D}_X$-presented over $Z\times \mathcal{S}$ such that $\mathcal{H}_{\mathcal{D}}^0(i):\mathcal{Y}^{cl}\rightarrow Z\times \mathcal{S}^{cl}$ is a $\mathcal{D}_X$-closed immersion. Then, since for every morphism of derived $\mathcal{D}_X$-schemes $\varphi:\mathcal{S}'\rightarrow \mathcal{S}$ there is an induced (by pull-back) morphism of $\infty$-groupoids,
$$\mathrm{PDE}_{\mathcal{D}_X}(Z)^{\mathrm{Tor}_0}(\mathcal{S})\xrightarrow{\varphi^*}\mathrm{PDE}_{\mathcal{D}_X}(Z)^{\mathrm{Tor}_0}(\mathcal{S}'),$$
we obtain the desired functor 
$\underline{\mathrm{PDE}}_{\mathcal{D}_X}(Z)^{\mathrm{Tor}_0}$ from $\mathbf{dSch}_{\mathcal{D}_X}\rightarrow \infty\mathbf{Grpd}.$ 

Now, let us prove (1). Notice that if $U'\rightarrow U$ is étale the induced map $U_{dR}'\rightarrow U_{dR}$ is schematic, étale and surjective. Thus yields a cover for the fppf topology. Similar reasoning is applicable for proper-surjective descent. Then noticing that fppf descent combines both Nisnevich and finite flat descent, we see that $\mathbb{R}\underline{\mathrm{PDE}}_X$ satisfies $h$-descent on $\mathrm{DGSch}_{aft}^{aff}$ i.e. descent with respect to the $h$-topology generated by Zariski covers and proper-surjective covers. The claim then follows noticing fppf-covers are in particular, $h$-covers.

To prove (2), consider $f:X\rightarrow Y$ in $\mathbf{PStk},$ and let $E_Y\in \mathbf{PStk}_{/Y}.$ 
We claim that there is an induced functor
$$f_{PDE}^*:\mathbb{R}\underline{\mathrm{PDE}}_{Y}(Z_Y)\rightarrow \mathbb{R}\underline{\mathrm{PDE}}_{X}(Z_X),$$
for $Z_Y:=p_{dR*}^Y(E_Y)$ and for some $Z_X,$ to be identified, which is moreover an equivalence when $f$ is formally étale. 
To this end, consider the commutative diagram,
\[
\adjustbox{scale=.85}{
\begin{tikzcd}
\mathrm{dStk}\times\mathrm{Fun}([1]\times \mathrm{dAff}^{op},\mathrm{Spc})\times\{Y\}\arrow[r,"q_*^Y"] \arrow[d,"f^*"] & \mathrm{dStk}\times\mathrm{Fun}([1]\times \mathrm{dAff}^{op},\mathrm{Spc})\times\{Y_{DR}\}\arrow[d,"f_{DR}^*"]
\\
\mathrm{dStk}\times\mathrm{Fun}([1]\times \mathrm{dAff}^{op},\mathrm{Spc})\times\{Y\times_{Y_{DR}}X_{DR}\}\arrow[r] & \mathrm{dStk}\times\mathrm{Fun}([1]\times \mathrm{dAff}^{op},\mathrm{Spc})\times\{X_{DR}\}
\end{tikzcd}}
\]
where we fiber over $\mathbf{PStk}.$ There is a corresponding base-change square, where we indicate all right-adjoints
\[
\begin{tikzcd}
 \mathbf{PStk}_{Y\times_{Y_{dR}}X_{dR}}\arrow[d,shift left=.75ex,"f_*"]\arrow[r, "p_{dR*}"] & \mathbf{PStk}_{X_{dR}} \arrow[d,shift left=.75ex,"f_{dR*}"]
\\
 \mathbf{PStk}_{Y}\arrow[u,shift left=.75ex,"f^*"] \arrow[r, "p_{dR*}^Y"] & \mathbf{PStk}_{Y_{dR}}\arrow[u,shift left=.75ex,"f_{dR}^*"]
\end{tikzcd}
\]
understanding the $(p_{dR}^*,p_{dR*})$-adjunctions for $X,Y.$ 
On the level of points, our functor is defined by the induced pull-back functor $f_{dR}^*:\mathbf{PStk}_{Y_{dR}}\rightarrow \mathbf{PStk}_{X_{dR}},$ by
$$f_{PDE}^*\big(\iota_Y:\mathcal{Y}_Y\hookrightarrow \mathrm{Jet}_{dR,Y}^{\infty}(E_Y)\big):=\big(\iota_X:f_{dR}^*(\mathcal{Y}_X)\hookrightarrow f_{dR}^*\mathrm{Jet}_{dR,Y}^{\infty}(E_Y)\big).$$
Set
$Z_X:=\mathrm{Jet}_{dR,X}^*(f^*E_Y).$ Then, for étale morphisms there are equivalences
$$f^*\circ p_{dR}^{Y*}\circ p_{dR*}^Y\simeq p_{dR}^{X*}\circ p_{dR*}^X\circ f^*,$$
acting from $\mathbf{PStk}_{Y}\rightarrow \mathbf{PStk}_{X}$.
The Beck-Chevalley-Lurie condition implies the equivalence of functors $f_{dR}^*\circ p_{Y,dR*}\simeq p_{X,dR*}\circ f^*,$ and thus the result.
\end{proof}
It is possible to include descent over $X$, although we do not need it.
\begin{rmk}
By Proposition \ref{prop: Moduli of PDE} (1), we may write
$\mathbb{R}\underline{\mathrm{PDE}}_{X}\simeq \underset{U\rightarrow X}{lim}\hspace{1mm}\mathbb{R}\underline{\mathrm{PDE}}_{U},$
and in particular, for (\ref{eqn: RPDE-DX}), via stackification, $X$ may be a smooth Deligne-Mumford stack over $k.$ Thus it makes sense to write
\begin{equation}
    \label{eqn: RPDE Tor0}
\mathbb{R}\mathrm{PDE}_{\mathcal{D}_X}^{\mathrm{Tor}_0}=\underset{U\rightarrow X}{lim}\hspace{1mm}\mathbb{R}\mathrm{PDE}_{\mathcal{D}_U}^{\mathrm{Tor}_0},
\end{equation}
 with the limit (\ref{eqn: RPDE Tor0}) taken over the small étale site of $X$ inside the $\infty$-category of presentable dg-categories.
 \end{rmk}

By assumption of being $\D$-finitary, we have that by pull-back along every solution for each test scheme $T$, giving a functor $ev_{u_T}^*:\mathbf{Perf}_{\D_X}(Z)\rightarrow \mathbf{Perf}(T\times X_{DR}),$ simply denoted here by $\varphi$ (see \cite[Proposition 6.45]{KSY}).

It is the pull-back (on perfect complexes, as per Assumptions \ref{Assumptions}), obtained for each test scheme $T$ from the diagram in $\mathbf{PStk}_{X_{DR}},$
\begin{equation}
    \label{eqn: SolDiagram}
\begin{tikzcd}
T\times X\arrow[d,"id_T\times q^X"]\arrow[rr, bend left, "ev_{u_T}'"] \arrow[r,"u\times id_X"] & Z\times_{X_{DR}}X\arrow[d]\arrow[r] & Z
\\
T\times X_{DR}\arrow[r,"ev(u_T)"] & Z& 
\end{tikzcd}
\end{equation}
The corresponding assignment will be denoted by
\begin{equation}
    \label{eqn: PhiL} \mathbb{L}_{\varphi}:=\varphi^*\mathbb{L}:\mathrm{PDE}_X(E)^-\rightarrow \mathbf{Perf}(\mathcal{D}_X).
    \end{equation}
    Assume that $X$ is a proper $k$-scheme of finite-type. Then for each test scheme $T$, $\mathbb{L}_{u_T}(Z)=\mathbb{L}_{Z/X_{DR},u_T}$ is as in (\ref{eqn: Tangent D}), and gives a sheaf of dg-$\mathcal{O}_T$-modules via the canonical push-forward along the map $a_{X_{DR}}:X_{DR}\rightarrow Spec(\mathbb{C}),$
$$(id_T\times a_{X_{DR}})_*:\mathbf{Perf}(T\times X_{DR})\rightarrow \mathbf{Perf}(T).$$

Since $X$ is smooth $\mathcal{E}_X^{\mathbb{R}}$ is well-defined sheaf of rings on $T^*X$, with $\mathcal{E}_X^{\mathbb{R}}-\mathrm{dg}_{T^*X}^{qcoh},$ the $\infty$-category of quasi-coherent sheaves on $T^*X$ of complexes of $\mathcal{E}_X^{\mathbb{R}}$-modules.
Via pull-back along $\varphi$ and tensorization, from (\ref{eqn: PhiL}) there is an associated functor denoted
$\varphi^*\mathbb{L}\otimes \mathcal{E}_X^{\mathbb{R}},$ acting by 
\begin{equation}
    \label{eqn: Lambda Localization}
    \mathrm{PDE}_X(E)^-\rightarrow \mathbf{Perf}(\mathcal{D}_X)\rightarrow \mathcal{E}_X^{\mathbb{R}}-\mathrm{dg}_{T^*X}^{qcoh}.
    \end{equation}
    

Suppose that $\Lambda\subset T^*X$ is a closed conical subset with
$\mathbf{QCoh}_{\Lambda}(\mathcal{D}_X)\subset \mathbf{QCoh}(\mathcal{D}_X),$ 
the full subcategory of quasi-coherent $\mathcal{D}$-modules $\mathcal{M}^{\bullet}$ for which 
$Char(\mathcal{M}^{\bullet})\subset \Lambda\subset T^*X.$

\begin{rmk}
    In localization terms, one considers $\mathbf{QCoh}_{\Lambda}(\mathcal{E}_X^{\mathbb{R}}),$
of $\mathcal{E}_X^{\mathbb{R}}$-modules localized by the null-system of coherent modules for which their support is not contained in $\Lambda.$
\end{rmk}

Consider the inclusion $j_{\Lambda}:T^*X\backslash \Lambda\hookrightarrow T^*X$ and the induced functor
$$j_{\Lambda}^*:\mathcal{E}_X^{\mathbb{R}}-\mathrm{dg}_{T^*X}^{qcoh}\rightarrow \mathcal{E}_X^{\mathbb{R}}-\mathrm{dg}_{T^*X\backslash \Lambda}^{qcoh},$$
with $\mathcal{E}_X-\mathrm{dg}_{\Lambda,T^*X}^{qcoh}\simeq\mathrm{hofib}(j_{\Lambda}^*).$
Combining (\ref{eqn: Lambda Localization}) with $j_{\Lambda}^*$ we may form the pull-back diagram
\begin{equation}
    \label{eqn: RPDEDiagram}
\begin{tikzcd}
    \mathbb{R}\underline{\mathrm{PDE}}_X^{\Lambda,\varphi}(\mathcal{Z})\arrow[d]\arrow[r] & \mathrm{hofib}(j_{\Lambda}^*)\arrow[d]
    \\
    \mathbb{R}\underline{\mathrm{PDE}}_X(\mathcal{Z})\arrow[r] & \mathcal{E}_X^{\mathbb{R}}-\mathrm{dg}_{T^*X}^{qcoh}.
\end{tikzcd}
\end{equation}
Thus for every solution $\varphi$, obtained via (\ref{eqn: SolDiagram}), by Assumptions \ref{Assumptions} the construction does not depend on the solution, and (\ref{eqn: RPDEDiagram}) defines a sub $\infty$-groupoid,
\begin{equation}
    \label{eqn: SubGrpd}
\mathbb{R}\underline{\mathrm{PDE}}_{X}^{\Lambda,\varphi}(\mathcal{Z}):=\mathbb{R}\underline{\mathrm{PDE}}_X(\mathcal{Z})\times_{\mathcal{E}_X^{\mathbb{R}}-\mathrm{dg}_{T^*X}^{qcoh}}\mathcal{E}_X^{\mathbb{R}}-\mathrm{dg}_{\Lambda,T^*X}^{qcoh}.
\end{equation}
Objects of (\ref{eqn: SubGrpd}) may be interpreted as PDEs (closed $\D$-subschemes in $\mathcal{Z}$), with
$Char_{\mathcal{D}}(\mathcal{Y})\subset \Lambda,$
since the support of $\varphi^*\mathbb{L}_{\mathcal{Y}}\otimes \mathcal{E}_X^{\mathbb{R}}$ is contained in $\Lambda.$
Omit reference to $\varphi$, write
$\mathbb{R}\underline{\mathrm{PDE}}_X^{\Lambda,*}(\mathcal{Z}).$

We may vary over $X$ by considering a morphism
$f:Z\rightarrow X$ and supposing that $\Sigma_Z\subset Z,\Sigma_X\subset X$ are closed sub-manifolds. Denote $f_{\Sigma}$ the restriction and assume it is a closed embedding. There is an induced microlocal diagram,
\begin{equation}
    \label{eqn: Microlocal Diagram}
\begin{tikzcd}
    T^*Z & \arrow[l,"j_d"] Z\times_XT^*X\arrow[d] \arrow[r,"j_{\pi}"] & T^*X
\\
T_{\Sigma_Z}^*Z\arrow[u] \arrow[d,"\pi_{\Sigma_Z}"] & \arrow[l,"j_d"] \Sigma_Z\times_{\Sigma_X}T_{\Sigma_X}^*X\arrow[d,"\pi"] \arrow[r,"j_{\pi}"] & T_{\Sigma_X}^*X\arrow[u]\arrow[d,"\pi_{\Sigma_X}"]
\\
\Sigma_Z \arrow[r,"id"] & \Sigma_X \arrow[r,"j"] & \Sigma_X
\end{tikzcd}
\end{equation}

As a special case in (\ref{eqn: Microlocal Diagram}), take $(\Sigma\subset X),$ and fix a $\mathcal{D}_X$-algebra $\mathcal{B},$ concentrated in degree zero.

\begin{defn}
\label{eqn: Elliptic Moduli}
\normalfont
The space of \emph{$\Sigma$-elliptic $X_{DR}$-spaces} is $\mathbb{R}\underline{\mathrm{PDE}}_{\Sigma\subset X}^{e}(\mathcal{Z}):=\mathbb{R}\underline{\mathrm{PDE}}_X^{\Sigma \times_X T_X^*X,*}(\mathcal{Z}).$
\end{defn}
Definition \ref{eqn: Elliptic Moduli} gives a dg-enhancement of \cite[Definition 5.11]{KSh}. Indeed, it states that $supp\big(\varphi^*\mathbb{L}_{\mathcal{Y}}\otimes\mathcal{E}_X^{\mathbb{R}}\big)\cap T_{\Sigma}^*X\subset \Sigma\times_X T_X^*X,$
holds for all $\varphi$ as the condition obtained via pulling back 
$$Char_{\mathcal{D}}(\mathcal{Y})\times \RS_X(\mathcal{Y})\cap T_{\Sigma}^*X\subset \RS_X(\mathcal{Y})\times_X \Sigma \times_X T_X^*X,$$
along each point $\varphi$ of $\RS_X(\mathcal{Y}).$
\subsubsection{$\D$-Geometric Grassmannians}
\label{ssec: D-Geometric Grassmannians}
Fix $\underline{k}=(0\leq k_1<k_2<\cdots<k_{\ell})$ for $\ell\geq 1.$
Let $\mathcal{M}^{\bullet}$ be a compact object which is connective i.e. $\mathcal{H}_{\mathcal{D}}^i(\mathcal{M}^{\bullet})=0,\forall i>0.$
Define the functor
$$\mathbb{R}\underline{\mathbf{Flag}}_{\mathcal{D},Z}:\mathbf{dSch}(\D_X)_{/Z}\rightarrow \mathbf{Spc}, \big(\eta:\mathcal{T}=\mathbb{R}\mathrm{Spec}_{\mathcal{D}}(\mathcal{C})\rightarrow Z\big)\mapsto \mathbb{R}\mathbf{Flag}_{\mathcal{D},Z}(\mathcal{T}),$$
where 
$$\mathbb{R}\mathbf{Flag}_{\mathcal{D},Z}(\mathcal{T})\subseteq \mathrm{Fun}(\Delta^{\ell},\mathbf{Mod}(\mathcal{O}_{\mathcal{T}}[\mathcal{D}_X])^{\leq 0}\big)^{\simeq},$$
spanned by:
\begin{equation}
    \begin{cases}
\EuScript{M}_{\mathcal{T}}^{\bullet}:=\big(\eta^*\mathcal{M}\xrightarrow{\varphi_{\ell}}M_{\ell}\rightarrow M_{\ell-1}\xrightarrow{\varphi_{\ell-1}}\cdots\xrightarrow{\varphi_0}M_1\big),
\\
\varphi_i,\text{ is surjective morphism of } \mathcal{H}_{\mathcal{D}}^0(\mathcal{O}_{\mathcal{T}})[\mathcal{D}_X]-\text{modules},
\\
M_i \text{ is a vector $\D$-bundle of } \mathrm{rank}(M_i)=k_i.
\end{cases}
\end{equation}
In particular, each $M_{\ell}$ will be induced as an $\mathcal{A}[\mathcal{D}]$-module of finite-presentation and its rank is given by the rank of the underlying $\mathcal{O}_X$-module.

The rank $k$-derived $\D$-Grassmannian over a $\mathcal{D}$-smooth $\mathcal{D}$-scheme $Z$, to be denoted by
$\mathbb{R}\underline{\mathbf{Grass}}_{\mathcal{D},Z}
(k,\mathcal{M}),$
is obtained straightforwardly by putting $\ell=1$ with $\underline{k}=(k).$ Its more explicit construction is now given by considering an analog of a classical Grassmannian adapted to the dg-$\D$-setting. Namely, for a perfect complex of $\D_X$-modules, set
$$\mathrm{Grass}_{\mathcal{D}}(\mathcal{M}^{\bullet}):=\mathrm{Grass}\big(\mathcal{H}_{\mathcal{D}}^*(\mathcal{M})\big)=\prod_j\mathrm{Grass}\big(\mathcal{H}_{\mathcal{D}}^j(\mathcal{M})\big).$$
In other words,
$$\mathrm{Grass}_{\mathcal{D}}(\mathcal{M}^{\bullet})=\prod_j\bigg(\bigsqcup_{i=0}^{\mathrm{dim}(\mathcal{H}_{\mathcal{D}}^j(\mathcal{M}))}\mathrm{Grass}\big(i;\mathcal{H}_{\mathcal{D}}^j(\mathcal{M}^{\bullet})\big)\bigg),$$
where 
$\mathrm{Grass}\big(i;\mathcal{H}_{\mathcal{D}}^j(\mathcal{M}^{\bullet})\big):=\big\{\mathcal{N}\subseteq \mathcal{H}_{\mathcal{D}}^j(\mathcal{M}^{\bullet})|\mathrm{dim}(\mathcal{N})=i\big\}.$
It represents the functor of points $\mathbf{Gr}\big(\mathcal{H}_{\mathcal{D}}^*(\mathcal{M}^{\bullet})\big):\mathrm{CAlg}_{X}(\mathcal{D}_X)\rightarrow Sets,$ mapping a classical $\D$-algebra $\mathcal{C}$ to the set of sub-objects 
$$\{\mathcal{N}\subset \mathcal{H}_{\mathcal{D}}^*(\mathcal{M})\otimes \mathcal{C}|\mathcal{N}\text{ cofibrant }\}.$$

\begin{cons}
    \label{cons: D-Prestack of Grassmannian}
    \normalfont 

Fix our $\D$-affine $k$-scheme $X$ such that $\mathcal{A}^{\bullet}$ is a derived $\D_X$-algebra for which $\mathbb{L}_{\mathcal{A}}$ exists and is perfect and dualizable with $\mathcal{A}[\mathcal{D}]$-dual $\mathbb{T}_{\mathcal{A}}^{\ell}$. In this case, we may locally write this complex as a complex of free $\mathcal{A}[\mathcal{D}]$-modules of finite rank. 
Consider the classical $\D$-prestack
$$\underline{\mathcal{G}\mathrm{rass}}_{\mathcal{D}}(\mathcal{A}):\mathrm{CAlg}_X(\mathcal{D}_X)\rightarrow \mathbf{Spc},$$
defined for each commutative $\D$-algebra $\mathcal{C}$ by the groupoid of $\mathcal{C}$-points of sequences:
$$\big\{\mathcal{N}^{\bullet}\hookrightarrow \mathcal{L}^{\bullet}\xrightarrow{\varphi} \mathbb{T}_{\mathcal{A}}^{\ell}\otimes\mathcal{C}| \mathcal{H}_{\mathcal{D}}^i(\mathcal{N}^{\bullet})\text{ is flat over } \mathcal{C}$$
$$\text{ and }\mathcal{H}_{\mathcal{D}}^i(\mathcal{N}^{\bullet})\rightarrow \mathcal{H}_{\mathcal{D}}^i(\mathbb{T}_{\mathcal{A}}^{\ell})\otimes 
 \mathcal{C} \text{ is an injective } \mathcal{D}-\text{morphism }\forall i\big\}.$$
\end{cons}

\begin{rmk}
If $\mathcal{A}^{\bullet}$ is a $\D$-smooth discrete commutative $\D$-algebra corresponding to a jet construction, then the inclusion $\mathcal{H}_{\mathcal{D}}^0(\mathcal{N})\hookrightarrow \mathcal{H}_{\mathcal{D}}^0(\Theta_{\mathcal{A}}^{\ell})\otimes \mathcal{C},$ is to be thought of as corresponding to the case of a parameterized natural inclusion 
$$T_{z}\mathrm{Spec}_{\mathcal{D}}(\mathcal{A}/\mathcal{I})\hookrightarrow F^*T_{z}\mathrm{Spec}_{\mathcal{D}}(\mathcal{A}),$$ of global sections of tangent $\D$-schemes arising from a $\D$-sub-scheme e.g. a morphism of $\D$-schemes $F:\mathrm{Spec}_{\mathcal{D}}(\mathcal{B})\rightarrow \mathrm{Spec}_{\mathcal{D}}(\mathcal{A}).$
\end{rmk}

We now construct a derived enhancement of the classical prestack given by Construction \ref{cons: D-Prestack of Grassmannian}.

To this end, consider the assignment for every perfect complex of $\mathcal{A}[\mathcal{D}]$-modules $\mathcal{M}^{\bullet}$, which we will end up taking to be $\mathbb{T}_{\mathcal{A}}^{\ell}$:
\begin{equation}
    \label{eqn: Pre Derived Grassmannian}
\mathbb{R}^{pre}\underline{\mathcal{G}\mathrm{rass}}_{\mathcal{D}}(\mathcal{M}^{\bullet})(-):\mathbf{CAlg}(\mathcal{D}_X)\rightarrow \mathbf{Spc},
\end{equation}
which prescribes to each $\D$-algebra $\mathcal{C}^{\bullet}$ the space
$$\mathbb{R}^{pre}\underline{\mathcal{G}\mathrm{rass}}_{\mathcal{D}}(\mathcal{M}^{\bullet})(\mathcal{C}^{\bullet}):=\big\{\text{Sequences } \mathcal{N}^{\bullet}\hookrightarrow \mathcal{R}^{\bullet}\xrightarrow{qis}\mathcal{M}^{\bullet}\otimes^{\mathbb{L}}\mathcal{C}^{\bullet}\text{ in } \DG(\mathcal{C}^{\bullet})^{cofib}\big\}.$$
The simplicial $\D$-presheaf (\ref{eqn: Pre Derived Grassmannian}) must be modified one step further to obtain a derived enhancement of $\mathcal{G}rass_{\mathcal{D}}(\mathcal{M}^{\bullet})$.

For each $\mathcal{C}^{\bullet}\in \mathbf{CAlg}(\mathcal{D}_X),$ there is a canonical morphism $\mathcal{C}^{\bullet}\rightarrow \mathcal{H}_{\mathcal{D}}^0(\mathcal{C}^{\bullet}),$ thus by 
direct comparison of 
 the definitions there is a canonical morphism
\begin{equation}
    \label{eqn: CanonIncl}
\underline{\mathcal{G}\mathrm{rass}}_{\mathcal{D}}(\mathcal{M}^{\bullet})\big(\mathcal{H}_{\mathcal{D}}^0(\mathcal{C}^{\bullet})\big)\rightarrow \mathbb{R}^{pre}\underline{\mathcal{G}\mathrm{rass}}_{\mathcal{D}}(\mathcal{M}^{\bullet})\big(\mathcal{H}_{\mathcal{D}}^0(\mathcal{C}^{\bullet})\big),
\end{equation}
which yields the following result.

\begin{prop}
Consider Construction \ref{cons: D-Prestack of Grassmannian} and the classical prestack $\underline{\mathcal{G}\mathrm{rass}}_{\D}(\A).$ 
Then, the morphism (\ref{eqn: CanonIncl}) induces a Cartesian diagram,
\begin{equation}
    \label{eqn: Grassmannian Fiber product}
    \begin{tikzcd}
\mathbb{R}\underline{\mathcal{G}}_{\mathcal{D}}(\mathcal{M}^{\bullet})(\mathcal{C}^{\bullet})\arrow[d]\arrow[r] & \underline{\mathcal{G}\mathrm{rass}}_{\mathcal{D}}(\mathcal{M}^{\bullet})\big(\mathcal{H}_{\mathcal{D}}^0(\mathcal{C}^{\bullet})\big)\arrow[d]
        \\
       \mathbb{R}^{pre}\underline{\mathcal{G}\mathrm{rass}}_{\mathcal{D}}(\mathcal{M}^{\bullet})(\mathcal{C}^{\bullet})\arrow[r] & \mathbb{R}^{pre}\underline{\mathcal{G}\mathrm{rass}}_{\mathcal{D}}(\mathcal{M}^{\bullet})\big(\mathcal{H}_{\mathcal{D}}^0(\mathcal{C}^{\bullet})\big) 
    \end{tikzcd}
\end{equation}
such that $\tau_0\mathbb{R}\underline{\mathcal{G}}_{\mathcal{D}}(\mathcal{M}^{\bullet})(\mathcal{C}^{\bullet})\simeq \underline{\mathcal{G}\mathrm{rass}}_{\D}.$
    \end{prop}
  \begin{proof}
Note that the fiber product (\ref{eqn: Grassmannian Fiber product}) is equivalent to the functor
$$\mathbb{R}\underline{\mathcal{G}\mathrm{rass}}_{\mathcal{D}}(\mathcal{M}^{\bullet}):\mathbf{CAlg}(\mathcal{D}_X)\rightarrow \mathbf{Spc},$$
    and note its $\mathcal{C}^{\bullet}$-points are given by sequences of the form 
    $$\mathcal{N}^{\bullet}\hookrightarrow \mathcal{L}^{\bullet}\rightarrow \mathcal{M}^{\bullet}\otimes\mathcal{C}^{\bullet},$$
with the property that upon tensoring with the commutative $\D$-algebra $\mathcal{H}_{\mathcal{D}}^0(\mathcal{C}^{\bullet})$ (over $\mathcal{C}$), given a point of $\mathbb{R}^{pre}\underline{\mathcal{G}\mathrm{rass}}_{\mathcal{D}}(\mathcal{M}^{\bullet})\big(\mathcal{H}_{\mathcal{D}}^0(\mathcal{C}^{\bullet})\big).$
Consider the classical truncation. We will show the functor of points is represented by the same object which represents the classical prestack  
 $\mathbf{Gr}\big(\mathcal{H}_{\mathcal{D}}^*(\mathcal{M}^{\bullet})\big)$ introduced above. This follows since $X$ is $\D$-affine one has the finite homological dimension of $\D$ as $2\mathrm{dim}_X$ so that a cofibrant module with coefficients in $\mathcal{A}$ is equivalently described in terms of a perfect complex over $\D_X$ by pull-back along zero-section morphism to $X$ since $\mathbf{D}_{\mathrm{Coh}}^b(\mathcal{D}_X^{op})\simeq \mathbf{Perf}(\mathcal{D}_X)$ is an equivalence of $\infty$-categories. In other words,
    $$\mathbb{R}\underline{\mathcal{G}\mathrm{rass}}_{\mathcal{D}}(\mathcal{M}^{\bullet})^{cl}(\mathcal{C})\simeq\big\{\mathcal{N}\hookrightarrow \mathcal{M}\otimes\mathcal{C}|\mathcal{N}\in \mathbf{Perf}(\mathcal{D}_X),\mathcal{H}_{\mathcal{D}}^*(\mathcal{N})\text{ is } \mathcal{C}-\text{flat and }$$
    $$\mathcal{H}_{\mathcal{D}}^*(\mathcal{N})\hookrightarrow \mathcal{H}_{\mathcal{D}}^*(\mathcal{M})\otimes\mathcal{C}\text{is  injective}\big\}/\{quasi-isom\}.$$
    The equivalence follows by noticing that for a point $[\mathcal{N}]\in \mathbf{Gr}\big(\mathcal{H}_{\mathcal{D}}^*(\mathcal{M}^{\bullet})\big)$, corresponding to $[\mathcal{N}\hookrightarrow \mathcal{H}_{\mathcal{D}}^*(\mathcal{M})\otimes\mathcal{C}],$ there exists a unique quasi-isomorphism class $[\mathcal{L}]\in \mathbb{R}\underline{\mathcal{G}\mathrm{rass}}_{\mathcal{D}}(\mathcal{M}^{\bullet})^{cl}(\mathcal{C})$ whose cohomology agrees with $[\mathcal{N}].$ This follows since $\mathcal{L}$ is $\mathcal{D}_X$-perfect with flat cohomology. In particular, it is a complex of locally free $\mathcal{D}_X$-modules in each degree.
\end{proof}

\subsection{The stack of $\D$-morphisms}
\label{ssec: An alternative description}
In this subsection we give another point of view on the $\D$-prestack \ref{eqn: DerDQuot}. This gives more details to prove Proposition \ref{prop: DerDQuot is stack}. 

\begin{prop}
The functor $\mathbb{R}\underline{Q}_{\mathcal{D}_X,Z}^{P},$ is a sub-prestack of $\mathbf{maps}_{\mathcal{D}_X}^{rel}(\mathcal{O}_{Z},-)^{\leq 0}.$
\end{prop}
\begin{proof}
The proof is based on the description of another $\mathcal{D}_X$-prestack defined as the simplicial $\mathcal{D}_X$-presheaf 
$$N\circ \underline{\mathbf{QCoh}}_{\mathcal{D}}^{W,cof}:\mathbf{CAlg}(\mathcal{D}_X)\rightarrow \mathbf{Spc},$$
given by composing the nerve functor with $\underline{\Q}_{\mathcal{D}}^{W,cof},$ the $\mathcal{D}_X$-prestack sending $\mathcal{A}^{\bullet}$ to the sub-category of $\Q_{\mathcal{D}}(\mathcal{A})$ consisting of weak-equivalences between cofibrant objects.
There is a bijection 
$$\pi_0\big(N\circ \underline{\Q}_{\mathcal{D}}^{W,cof}(\mathcal{A}^{\bullet})\big)\simeq \mathrm{Isom}\big(Ho(\mathbf{Mod}(\mathcal{A}[\mathcal{D}])\big),$$
where we have isomorphism classes of objects in the $\infty$-category of $\mathcal{A}$-modules in perfect $\mathcal{D}$-modules on $X.$

We obtain a sub-simplicial set, consisting of all connected components corresponding to those compact (as $X$ is smooth, thus perfect) objects in $Ho(\mathcal{A}^{\bullet}-\DG),$ given by $Ho\big(\mathcal{A}^{\bullet}-\mathbf{Perf}(\mathcal{D}_X)\big),$ with $\mathbf{Perf}(\mathcal{D}_X)$ the perfect complexes of $\mathcal{D}$-modules on $X:$
\[
\begin{tikzcd}
    \mathbf{Perf}_{\mathcal{D}_X}(\mathcal{A}^{\bullet})\arrow[d]\arrow[r] & Isom\big(Ho(\mathcal{A}^{\bullet}-\mathbf{Perf}(\mathcal{D}_X)\big)\arrow[d]
    \\
    N\Q_{\mathcal{D}}^{W,cof}(\mathcal{A}^{\bullet})\arrow[r] & \pi_0\big( N\Q_{\mathcal{D}}^{W,cof}(\mathcal{A}^{\bullet})\big)
\end{tikzcd}
\]

More directly, we can directly understand the functor 
$$\underline{\mathbf{Perf}}_{\mathcal{D}_X}:\mathbf{CAlg}(\mathcal{D}_X)\rightarrow \infty\mathbf{Grpd},$$
defined in the usual way. 
The mapping stack $\mathbf{Maps}(\mathcal{Y},\underline{\mathbf{Perf}}_{\mathcal{D}_X})\simeq \mathbf{Perf}_{\mathcal{D}_X}(\mathcal{Y}):\mathbf{CAlg}(\mathcal{D}_X)\rightarrow \infty\mathbf{Grpd},$ sends $\mathcal{B}^{\bullet}\mapsto \mathbf{Perf}_{\mathcal{D}_X}\big(\mathcal{Y}\times\mathbb{R}\mathrm{Spec}_{\mathcal{D}_X}(\mathcal{B}^{\bullet})\big)^{\simeq}.$

Straightforwardly define sub-$\mathcal{D}_X$-prestacks  $\mathbf{Perf}_{\mathcal{D}_X}(\mathcal{A}^{\bullet})^{[a,b]}\subset \mathbf{Perf}_{\mathcal{D}_X}(\mathcal{A}^{\bullet}),$ of those complexes of amplitude $[a,b]$. We also have a well-defined moduli of perfect $\mathbf{Perf}_{\mathcal{D}_X}(\mathcal{A}^{\bullet}\otimes \mathcal{B}^{\bullet})$ consisting of those complexes properly (micro)supported over $Z\times \mathbb{R}\mathrm{Spec}_{\mathcal{D}_X}(\mathcal{B}^{\bullet})$ relative to $\mathbb{R}\mathrm{Spec}_{\mathcal{D}_X}(\mathcal{B}^{\bullet}).$ 
Denote it by 
$$\mathbf{Perf}_{\mathcal{D}_X}^{rel}(Z)\subset \mathbf{Perf}_{\mathcal{D}_X}(Z).$$

Consider the $\D$-presheaf $\mathbf{maps}_{\mathcal{D}_X}(\mathcal{O}_{Z},-):\mathbf{CAlg}(\mathcal{D}_X)\rightarrow \infty\mathbf{Grpd},$ assigning to $\mathcal{B}^{\bullet}$ the maximal sub $\infty$-groupoid of $\mathcal{O}_{Z}\otimes\mathcal{B}^{\bullet}-\mathbf{Perf}(\mathcal{D}_X).$
Then, the stack of morphisms with source $\mathcal{O}_{Z},$ whose targets are perfect complexes properly microsupported relative to the second factor is obtained by considering the the homotopy fiber product
\begin{equation}
    \label{eqn: Perfect Morphisms}
    \mathbf{maps}_{\mathcal{D}_X}^{rel}(\mathcal{O}_{Z},-)\simeq \mathbf{maps}_{\mathcal{D}_X}(\mathcal{O}_{Z},-)\times_{\mathbf{Perf}(Z)}\mathbf{Perf}_{\mathcal{D}_X}^{rel}(Z),
\end{equation}
as well as 
\begin{equation}
    \label{eqn: Perfect [a,b]- Morphisms}
    \mathbf{maps}_{\mathcal{D}_X}^{rel}(\mathcal{O}_{Z},-)^{[a,b]}\simeq \mathbf{maps}_{\mathcal{D}_X}(\mathcal{O}_{Z},-)\times_{\mathbf{Perf}(Z)}\mathbf{Perf}_{\mathcal{D}_X}^{rel}(Z)^{[a,b]},
\end{equation}
where $\mathcal{E}^{\bullet}\in \mathbf{Perf}_{\mathcal{D}_X}^{rel}(Z)^{[a,b]}$ is a perfect complex such that for every $\mathcal{N}\in \mathcal{H}_{\mathcal{D}}^0(\mathcal{O}_{Z})-\mathrm{mods}_{\mathcal{D}_X},$ one has that $\mathcal{H}_{\mathcal{D}}^i\big(\mathcal{E}^{\bullet}\otimes_{\mathcal{O}_{Z}}^{\mathbb{L}}\mathcal{N}\big)\simeq 0,$ for $i\neq [a,b].$
Thus from (\ref{eqn: Perfect [a,b]- Morphisms}) have for all $a\leq b$ a morphism
\begin{equation}
    \label{eqn: gamma}
\gamma_{a,b}:\mathbf{maps}_{\mathcal{D}_X}^{rel}(\mathcal{O}_{Z},-)^{[a,b]}\rightarrow \mathbf{Perf}_{\mathcal{D}_X}^{rel}(Z)^{[a,b]}.
\end{equation}

\begin{sublem}
\label{sublem: Sublemma1}
$^{cl}\mathbb{R}Q_{\mathcal{D}_X,Z}\subset \hspace{1mm}^{cl}\mathbf{maps}_{\mathcal{D}_X}^{rel}(\mathcal{O}_{Z},-)^{\leq 0},$ is a Zariski open $\mathcal{D}_X$-scheme (viewed as derived $\mathcal{D}$-stacks).
\end{sublem}
\begin{proof}
Considering the Kan extension of the restriction of the functor (\ref{eqn: DerDQuot}) to the full-subcategory $Sch_X(\mathcal{D}_X)\subset \mathbf{dSch}_{X}(\mathcal{D}_X),$ it suffices to consider classically parameterized points i.e. $\varphi:\mathcal{O}_{Z}\otimes^{\mathbb{L}}\mathcal{B}\rightarrow \mathcal{E},$ where $\mathcal{B}$ is classical as a $\mathcal{D}_X$-algebra. Then, one can show that 
$$^{cl}\mathbb{R}Q_{\mathcal{D}_X,Z}\times_{^{cl}\mathbf{maps}_{\mathcal{D}_X}^{rel}(\mathcal{O}_{Z},-)}\mathrm{Spec}_{\mathcal{D}}\big(\mathcal{H}_{\mathcal{D}}^0(\mathcal{B})\big)\rightarrow \mathrm{Spec}_{\mathcal{D}_X}\big(\mathcal{H}_{\mathcal{D}}^0(\mathcal{B})\big),$$
is a Zariski open immersion of $\mathcal{D}_X$-schemes. By Proposition \ref{prop: Zariski substacks} and Proposition \ref{prop: D-epi is local}, this corresponds to the locus over which our complexes are $\mathcal{D}_X$-flat and whose parameterizing quotient morphism is a surjective morphism. Indeed, any $\mathrm{Spec}_{\mathcal{D}_X}(\mathcal{B})$-point determines some quotient map $\varphi$ to $\mathcal{C},$ concentrated in degree zero. The conditions on Tor-amplitudes here then imply that the space of $\mathcal{O}_{Z}\otimes\mathcal{B}[\mathcal{D}_X]$-linear morphisms from $\mathcal{C}$ to $\mathrm{ker}(\varphi)$ has Tor amplitude $0$. Therefore it is flat. 
\end{proof}
This concludes the proof.
\end{proof}

Consider the morphisms (\ref{eqn: gamma}).
\begin{prop}
    There exists a sequence of natural transformations $\gamma_{a,b}\rightarrow \gamma_{a,b-1},$ whose corresponding filtered colimit of open substacks is a homotopically finitely presented derived $\mathcal{D}_X$-prestack, which is locally geometric. 
\end{prop}
\begin{proof}
Let $\varphi_{\mathcal{E}}:\mathbb{R}\mathrm{Spec}_{\mathcal{D}_X}(\mathcal{B})\rightarrow \mathbf{Perf}_{\mathcal{D}_X}^{rel}(Z)^{[a,b]},$ be a point classifying a perfect complex $\mathcal{E}^{\bullet}$ on $Z\times^h\mathbb{R}\mathrm{Spec}_{\mathcal{D}_X}(\mathcal{B}),$ we consider the pull-back diagram in derived $\mathcal{D}_X$-stacks:
\[
\begin{tikzcd}
\mathbb{R}\underline{\mathbf{maps}}_{\mathcal{D}_X}^{rel}(\mathcal{O}_{Z},\mathcal{E}^{\bullet})\arrow[d]\arrow[r] & \mathbb{R}\underline{Spec}_{\mathcal{D}_X}(\mathcal{B})\arrow[d]
  \\
\mathbf{maps}_{\mathcal{D}_X}(\mathcal{O}_{Z},-)^{[a,b]}\arrow[r,"\gamma_{a,b}"] & \mathbf{Perf}_{\mathcal{D}_X}^{rel}(Z)^{[ab]}
\end{tikzcd}
\]
where
$\mathbb{R}\underline{\mathbf{maps}}_{\mathcal{D}_X}^{rel}(\mathcal{O}_{Z},\mathcal{E}^{\bullet}):\mathbf{CAlg}(\mathcal{D}_X)_{\mathcal{B}/}\rightarrow \infty\mathbf{Grpd},$
assigns to $f:\mathcal{B}^{\bullet}\rightarrow \mathcal{S}^{\bullet}$ (with opposite morphism $\overline{f}$) the mapping space $$\mathrm{Maps}_{/Z\times^h\mathbb{R}\mathrm{Spec}_{\mathcal{D}_X}(\mathcal{B})}\big(\mathcal{O}_{Z}\otimes^{\mathbb{L}}\mathcal{S}^{\bullet},(id_{Z}\times\overline{f})^*\mathcal{E}^{\bullet}).$$
Assuming $Z$ is homotopically finitely $\mathcal{D}$-presented as a derived $\mathcal{D}_X$-stack, we have that 
$\mathcal{O}_{Z}$ equivalent to a retract of a finite cellular $\mathcal{D}_X$-algebra (see \cite[Proposition. 4.3.1]{KSY}).
Therefore, we have a sequence
$\mathcal{R}_1\rightarrow \cdots\rightarrow \mathcal{R}_{i}\rightarrow\mathcal{R}_{i+1}\rightarrow\ldots \rightarrow \mathcal{R}_n=\mathcal{O}_{Z}',$ inducing a sequence of mapping stacks, with $i$-th term for example, given by 
$$Maps_{/Z_{\mathcal{B}}}(\mathcal{R}_{i+1}^{\bullet}\otimes^{\mathbb{L}}\mathcal{S}^{\bullet},(id_{Z}\times \overline{f})^*\mathcal{E}^{\bullet})\rightarrow Maps_{/Z_{\mathcal{B}}}(\mathcal{R}_{i}^{\bullet}\otimes^{\mathbb{L}}\mathcal{S}^{\bullet},(id_{Z}\times \overline{f})^*\mathcal{E}^{\bullet}).$$
\end{proof}

\section{$\mathcal{D}$-Quot schemes as DG-$\D$-schemes}
\label{sec: D-Quot schemes as DG-D-schemes}
In this section we prove the representability of the derived $\D$-Quot functor (c.f. Notation \ref{notate: Special notation}) by an object of the category of differential-graded $\D$-schemes, as opposed to derived $\D$-schemes. These notions are not the same, but are related (e.g. \cite[Thm. 8]{BKS} and \cite{BKSY2}). Throughout this section we make use of this relationship. 

\subsection{Differential graded $\D$-Schemes}
Recall that $X$ is a smooth $D$-affine variety. 
\begin{defn}
\label{defn: DG-D-Scheme}
\normalfont 
A \emph{differential-graded $\D$-scheme} is a pair $\mathbf{Z}=(Z_0,\mathcal{O}_{\mathbf{Z}}^{\bullet}),$ where $(Z_0,\mathcal{O}_{\mathbf{Z}}^0)$ is an ordinary $\D$-scheme, called the \emph{ambient classical $\D$-scheme}, and $\mathcal{O}_{\mathbf{Z}}^{\bullet}$ is a sheaf of differential-graded commutative $\D$-algebras. It is of $\D$-\emph{finite-type} if $(Z_0,\mathcal{O}_{\mathbf{Z}}^0)$ is of $\D$-finite type as a classical $\D$-scheme and each $\mathcal{O}_{\mathbf{Z}}^i$ is an $\mathcal{O}_{\mathbf{Z}}^0[\mathcal{D}_X]$-module of finite presentation.
    \end{defn}
A morphism $\mathbf{F}:=(F,f^{\bullet}):\mathbf{Z}_1\rightarrow\mathbf{Z}_2,$ of dg-$\D$-schemes is naturally defined as a morphism of dg-ringed spaces restricting to a morphism $(F,f^0)$ between the ambient classical $\D$-schemes. The corresponding category is denoted by $\mathrm{dgSch}_{X}(\mathcal{D}_X).$

    Following \cite{CFK}, put $\pi_0(\mathbf{Z}):=\mathrm{Spec}_{\mathcal{D}}(\mathcal{H}_{\mathcal{D}}^0(\mathcal{O}_{\mathbf{Z}}^{\bullet})\big),$ and call it the $0$-th truncation, reserving the term `classical truncation' for derived $\D$-schemes $Z$, denoted by $Z^{cl}.$
    
Underlying Definition \ref{defn: DG-D-Scheme} is a diagram
\begin{equation}
    \label{eqn: DG-D-Scheme Diagram}
\begin{tikzcd}
    & & \mathrm{Coker}(d^{-1})\arrow[hookrightarrow]{d} & 
    \\
    \cdots \rightarrow \mathcal{O}_{\mathbf{Z}}^{-2}\arrow[r,"d^{-2}"] & \mathcal{O}_{\mathbf{Z}}^{-1}\arrow[twoheadrightarrow]{ur} \arrow[r,"d^{-1}"] & \mathcal{O}_{\mathbf{Z}}^0 \rightarrow \cdots & 
\end{tikzcd}
\end{equation}
We only consider dg-schemes for which $\mathcal{O}_{\mathbf{Z}}^{i}=0,i>0.$ From (\ref{eqn: DG-D-Scheme Diagram}) it is clear that $\iota:\pi_0(\mathbf{Z})\subset Z_0=\mathrm{Spec}_{\mathcal{D}}(\mathcal{O}_{\mathbf{Z}}^0)$ defines a closed $\D$-subscheme.
    
To each dg-$\mathcal{D}$-scheme $\mathbf{Z}$ we may extract its underlying \emph{graded} $\D$-scheme:
$$\mathbf{Z}_{gr}:=(\pi_0(\mathbf{Z}),\iota^*\mathcal{H}_{\mathcal{D}}^*(\mathcal{O}_{\mathbf{Z}})\big),$$
and denoting by $\mathrm{Sch}_X^{gr}(\mathcal{D}_X),$ the subcategory of $\mathrm{dgSch}_{X}(\mathcal{D}_X)$ consisting of graded $\D$-schemes, we have a natural functor
$$(-)_{gr}:\mathrm{dgSch}_{X}(\mathcal{D}_X)\rightarrow \mathrm{Sch}_X^{gr}(\mathcal{D}_X),\hspace{2mm}\mathbf{Z}\mapsto \mathbf{Z}_{gr}.$$

\begin{defn}
\label{defn: dg-D-Manifold}
\normalfont 
    A dg-$\D$-scheme $\mathbf{Z}$ of $\D$-finite-type is said to be a \emph{dg-$\D$-manifold} if each $\mathcal{O}_{\mathbf{Z}}^i$ is a vector $\D$-bundle over the ambient $\D$-scheme. 
\end{defn}
There is a notion of $\D$-étale covering $(F,f^{\bullet})$ of dg-$\D$-manifolds, similar to that of Subsect. \ref{sssec: D-étale}. It is specified by asking that $\pi_0(\mathbf{Z}_1)\rightarrow \pi_0(\mathbf{Z}_2)$ is a usual $\D$-étale morphism and a $\D$-epimorphisms (\ref{eqn: Classical D-étale}), such that for each $k<0,$ there are isomorphism,
$$\mathcal{H}_{\mathcal{D}}^k(F^{-1}\mathcal{O}_{\mathbf{Z}_2}^{\bullet})\otimes_{F^{-1}\mathcal{H}_{\mathcal{D}}^0(\mathcal{O}_{Z_2}^{\bullet})}\mathcal{H}_{\mathcal{D}}^0(\mathcal{O}_{\mathbf{Z}_1}^{\bullet})\simeq \mathcal{H}_{\mathcal{D}}^k(\mathcal{O}_{\mathbf{Z}_1}^{\bullet}).$$
Equivalently, the base-change map from an affine derived $\D$-scheme is derived $\D$-étale.

\subsection{Differential graded `operations'}
In this subsection we discuss the following derived moduli problem: selecting $m_0\in \mathbb{N}$ large enough to ensure involutivity of $\EQ$, we want to classify $A_{\infty}$-submodules obtained via a certain category of $\D_X$-modules, of the form 
$$\mathcal{N}_{\geq m_0}\hookrightarrow \mathcal{C}h_{\geq m_0},\hspace{2mm} with \hspace{2mm} \mathrm{dim}(\mathcal{N}_{k})=\mathbf{h}^{\mathcal{D}}(k),\forall k\geq m_0.$$ 

Requiring existence of an $\mathcal{R}$-submodule $\mathcal{N}_{a,t}\hookrightarrow \mathcal{C}h$ is expressed via algebra relations on a product,
$$\mathrm{Gr}_{[a,t]}:=\prod_{a\leq s\leq t}\mathrm{Gr}\big(\mathrm{dim}(\mathcal{N}_s),\mathcal{C}_s\big).$$
Put $\alpha\leq\beta$ and $s\geq \alpha$ with $\mathrm{G}_s:=\mathrm{Gl}(P(s))$ with coefficients in the constant differential ring $\mathbb{C}$. Since a symbol module is finitely-cogenerated if the equation is involutive and formally integrable, we can choose a basis in $\mathcal{N}_s\hookrightarrow\mathcal{C}h,$ and this generates left actions of $\mathrm{G}_s$ on $\mathcal{N}_s.$

Here, $M_*:=\mathcal{C}h(*)$ is the graded $R_*:=\mathcal{O}_{T^*X}(*)$-module corresponding to an involutive $\D$-geometric PDE. Our notations parallel closely that used in \cite{BKSY2} to classify these algebra structures. The supplementary material on operadic structures (up-to-homotopy) for our use are given now.

The homotopy coherent mapping complexes in $\mathcal{A}[\mathcal{D}]$-modules are described in terms of a structure that Beilinson-Drinfeld refer to as `pseudo-tensor' \cite{BD}. We do not require the full machinery, but recall some features as they are needed. 
First we fix notation.

\begin{notate}
\normalfont
\label{notate: DG-operations}
    Let $X$ be a $\D$-affine $k$-scheme and consider for each non-negative integer $n$ the finite ordered set $[n]=\{1,\ldots,n\}$ and the complex of local $n$-array operations 
$$\mathbb{R}\mathcal{P}_{[n]}^*(\{\mathcal{M}_i\}_{i\in [n]};\mathcal{N}),$$
 between a collection of graded $\mathcal{A}[\mathcal{D}]$-modules $\mathcal{M}_i,\mathcal{N},i=1,\ldots,n.$ 
We will consider $\mathcal{R}:=\mathcal{O}_{T^*(X,\EQ)}$ with irrelevant ideal $\mathcal{R}_+$ and look at the corresponding multi-operations on (graded) microlocalized modules, for example the dg-version of the characteristic sheaf i.e. $\mathcal{M}:=\mu\mathbb{T}_{\mathcal{A}}^{\ell}$ in homogeneous degree $0$, which we denote by the same; for $n\geq 0,$
$$\mathcal{P}_n^*(\{\mathcal{M}\},\mathcal{N})^0, \hspace{3mm} \big(\text{resp. } \hspace{2mm} \mathcal{P}_{n\sqcup [1]}(\{\mathcal{M}_i,\mathcal{L}\}_{i\in [n]},\mathcal{N})^0\big),$$
indicate the spaces of $n$-array $\D$-linear morphisms (resp. $n+1$-array $\D$-linear morphisms) of homogeneous degree $0$ acting from 
$$\mathcal{M}_1\boxtimes \cdots\boxtimes\mathcal{M}_n\rightarrow \Delta_*\mathcal{N},\hspace{2mm} (\text{resp. } (\mathcal{M}_1\boxtimes \cdots\boxtimes \mathcal{M}_n)\boxtimes \mathcal{L}\rightarrow \Delta_*\mathcal{N})).$$
\end{notate}

One way to justify their use is due to the fact the `correct' duality in dg-$\mathcal{A}[\mathcal{D}]$-modules is not the $\mathcal{A}$-duality but rather the $\mathcal{A}[\mathcal{D}]$-dual (c.f. \ref{ssec: Background on derived D-geometry}). This is principally because $\mathcal{A}$-linear duality does not  have good finiteness properties. Indeed, the tangent $\D$-complex 
$\mathbb{T}_{\mathcal{A}}$ injects into $ \mathbb{R}\mathcal{H}om_{\mathcal{D}}(\mathcal{A},\mathcal{A}[\mathcal{D}])$ and $\D$-module theory tells us this is equivalently a map of $\D_{X\times X}^{op}$-linear dg-modules 
$\mathbb{T}_{\mathcal{A}}\boxtimes \mathcal{A}^r\rightarrow \Delta_*\mathcal{A}^r.$ Thus to study pure sheaves of $\mathcal{A}[\mathcal{D}]$-modules defined via the property of their reflexive local duality, one needs to work with these new kinds of operations. 
\\

\noindent\textbf{Digression: DG-operations.}
The usual operations we are interested in are recalled as follows (similar for $\mathcal{A}[\mathcal{D}_X]$-modules and $\mathcal{A}[\mathcal{E}_X]$-modules).
\begin{prop}
For $\mathcal{P}_{[1]}^*\big(\{\mathcal{M}\},\mathcal{N}\big):=\mathcal{H}\mathrm{om}_{\mathcal{D}_X}\big(\mathcal{M},\mathcal{N}\otimes^!\mathcal{D}_X\big)$ there is a natural evaluation map $\mathrm{eval}_*\in \mathcal{P}_{[2]}^*\big(\big\{\mathcal{M},\mathcal{P}_{[1]}^*(\{\mathcal{M}\},\mathcal{N})\big\},\mathcal{N}\big)$. More generally, there is an induced map $\mathcal{P}_{[I]}^*\big(\{\mathcal{L}_i\},\mathcal{P}_{[1]}^*(\{\mathcal{M}\},\mathcal{N})\big)\rightarrow \mathcal{P}_{[I]\sqcup \{1\}}^*\big(\{\mathcal{L}_i,\mathcal{M}\},\mathcal{N}\big)$ that is a bijection.
\end{prop}
In particular, there are internal homomorphism objects for $*$-operations and their $*$-compositions,
$$\mathcal{P}_{[1]}^*\big(\{\mathcal{M}\},\mathcal{N}\big)\boxtimes \mathcal{P}_{[1]}^*\big(\{\mathcal{N}\},\mathcal{L}\big)\rightarrow \Delta_*\mathcal{P}_{[1]}^*\big(\{\mathcal{M}\},\mathcal{L}\big).$$
The differential graded versions are defined for each finite set $I,$ and complexes $\mathcal{M}_i^{\bullet}$ and $\mathcal{N}^{\bullet}$ of $\mathcal{A}[\mathcal{D}]$-modules, by understanding
$$P_I(\{\mathcal{M}_i^{\bullet}\}_{i\in I},\mathcal{N}^{\bullet}),$$
to be the total complex associated with $P_I(\{\mathcal{M}_i^{\alpha_i}\}_{i\in I},\mathcal{N}^{\beta}),$ for all $\alpha_i,\beta\in \mathbb{Z}$.

Thus, the complex 
$\big(P_I\big(\{\mathcal{M}_i\}_{i\in I},\mathcal{N}\big),\partial(I),\big),$ is defined and for each surjective morphism of finite sets $\alpha\colon J\rightarrow I$ we have a composition map, which on homogeneous components looks like
\begin{eqnarray*}
    \gamma:P_I\big(\{C_i\},M\big)^m\otimes\bigotimes_{i\in I}P_{J_i}\big(\{D_j\}_{j\in J_i},C_i\big)^{n_i}&\rightarrow &P_J\big(\{D_j\}_{j\in J},M\big)^{m+\sum n_i},
    \\
   \big(f,(g_i)\big)&\longmapsto& \big(f(g_i)\big).
   \end{eqnarray*}
As a morphism of complexes, this composition interacts with the differential accordingly, via   
$$\partial\big(f(g_i)\big)=(\partial f)(g_i)+\sum_{k=1}^{|I|}(-1)^{m+\sum_{i=1}^{k-1}n_i}f\big(g_1,..,g_{k-1},\partial g_k,g_{k+1},\cdots,g_{|I|}\big).$$
 Explicitly it is given by
    $$\boxtimes_{j\in J} D_j=\boxtimes_{i\in I}\boxtimes_{j\in J_i}D_j\rightarrow \boxtimes_{i\in I}\Delta_*^{(J_i)}C_i=\Delta_*^{\alpha}\big(\boxtimes_{i\in I}C_i\big)\xrightarrow{\Delta^{\alpha}(f)}\Delta_*^{\alpha}\big(\Delta_*^IM\big),$$
    which is $\Delta_*^JM,$
    where we use that $\Delta^{\alpha}:=\prod_{i\in I}\Delta^{J_i}\colon X^I\hookrightarrow X^J$ and where $\Delta^{\alpha}\circ\Delta^I$ is identified with $\Delta^J.$
    We further impose the requirements that this composition rule satisfies:
    \begin{list}{$\bullet$}{}  
\item $($Associativity$)$: for $H\rightarrow J$ a morphism, with $\{W_h\}_{h\in H}$ an $H$-family of objects, with $h_j\in P_{H_j}\big(\{W_h\}_{h\in H_j}, D_j\big)^{p_j},$ we have that
   $$f\big(g_i(h_j)\big)=(-1)^{\sum_{i-1}^{|I|}\big(\sum_{\ell=1}^{i-1}(\sum_{j\in J_{\ell}}p_j)\big)n_i}\big(f(g_i)\big)(h_j)\in P_H^{\mathbf{C}}\big(\{W_h\},M\big);$$
   
\item $($Unitality$)$: for all objects $E$ there is an identity $1$-operation $id_E\in P_1\big(\{E\},E)^0,$ of degree $0$ such that $\partial id_E=0,$ and such that for all other operations $f\in P_I\big(\{C_i\},D\big)^n,$ we have
    $id_D(f)=f(id_{C_i})=f,$ for all $i.$
\end{list}
  
The associativity of the composition is phrased by the commutativity of the following diagram:
\[
\adjustbox{scale=.75}{
\begin{tikzcd}
P_I\big(\{C_i\},M\big)\otimes \bigotimes_{i\in I}\bigg(P_{J_i}\big(\{D_j\},C_i\big)\otimes_{j\in J_i}P_{K_j}\big(\{W_k\},D_j\big)\bigg)\arrow[d,] \arrow[r,] & P_I\big(\{C_i\},M\big)\otimes \bigotimes_{i\in I} P_{K_i}^{\mathbf{C}}\big(\{W_k\},D_j\big)
\arrow[dd,] 
\\
P_I\big(\{C_i\},M\big)\otimes \bigotimes_{i\in I} P_{J_i}\big(\{D_j\},C_i\big)\otimes \bigotimes_{j\in J}P_{K_j}\big(\{W_k\},D_j\big) \arrow[d,] 
\\
P_J\big(\{D_j\},M\big)\otimes  \bigotimes_{j\in J}P_{k_j}\big(\{W_k\},D_j\big)\arrow[r,] &  M
\end{tikzcd}}
\]
As $C_i,M$ carry a cohomological grading, the shifts are related by
$$P_I^{\mathbf{C}}\big(\{C_i[n_i]]\},M[m]\big)\cong P_I^{\mathbf{C}}\big(\{C_i\},M\big)\big[m-\sum_{i=1}^{|I|}n_i\big].$$
Spaces of operations are functorial in their arguments and for a given dg-category, induce a pseud-tensor structure in the associated homotopy-category.
\begin{prop}
\label{prop: Ho-pt-cat}
The DG-pseudo tensor structure $\mathcal{P}_I^{\mathbf{C}}$ induces a pseudo-tensor structure on the corresponding homotopy category, denoted $P_I^{H^0(\mathbf{C})},$ defined by 
$$P_I^{H^0(\mathbf{C})}(\{C_i\},D):=H^0\big(\mathcal{P}_I^{\mathbf{C}}(\{C_i,D)^{\cdot}\big).$$
\end{prop}
We use these generalized spaces of multilinear operations to classify algebraic structures in the $\D$-geometric setting.

\subsection{Classifying operations}
\label{sec: Classifying Operations}
A related example of a $*$-algebraic structure appears in the representation theory of the $L_{\infty}$-operad in $\A^r[\D]^{op}$-modules, where it arises from a cocommutative coassociative coalgebra object in complexes of such modules. When the components are perfect (i.e., dualizable), one may apply the Verdier duality (\ref{eqn: LocVerd}) to recover the corresponding $L_{\infty}^*$-algebroid structure via the Chevalley–Eilenberg formalism. An analogous construction holds for complexes of $\pi^*\A[\mathcal{E}_X]$-modules, where the $n$-ary operations are $\pi^*\A[\mathcal{E}_{X^n}]$-morphisms. The microlocal Verdier duality map (\ref{eqn: MicroLocalVerdier}) is an example of such a unary operation.

The dg-enhancements in the $\D$-setting are constructed via derived moduli of graded 
$\A[\D]$-submodules, using a chiral or differential refinement of Hochschild cohomology. This perspective is naturally phrased using operads \cite{LoVa}. We briefly recall the relevant features in the framework of (co)algebras over (co)operads, following \cite[§3]{FG}.

\subsubsection{Algebras over operads}
We now generalize the construction of the derived scheme
$\mathbb{R}\mathbf{Act}_{\mathcal{D}_X}\big(\mathcal{A}^{\bullet},\mathcal{M}^*\big),$ in \cite{DGQuot} to the present context in order to classify $\mathcal{A}_{\infty}^*$-actions of $\mathcal{A}^{\bullet}\in \mathbf{CAlg}(\mathcal{D}_X),$ on a (differential) graded $\D_X$-module $\mathcal{M}^*.$ 

In the case of (dg)-algebras and vector spaces, one puts 
$$\mathbb{R}\mathbf{Act}_{dg-\mathbb{C}}(A^{\bullet},V):=\mathrm{Act}_{dg-\mathbb{C}}(Q(A^{\bullet}),V),$$
using a cofibrant replacement or quasi-free replacement $Q(A^{\bullet})$ of $A^{\bullet}.$
To appropriately classify (sub)-module structures, a particular resolution is used, arising from the Bar construction.
Then,
$$\mathbb{R}\mathbf{Act}'(A^{\bullet},V):=\mathrm{Act}\big(\mathrm{BarCobar}(A),V\big),$$
using the quasi-free resolution of $A^{\bullet}$ given by the bar resolution of the co-bar resolution of $A^{\bullet}.$

\begin{rem}
\label{rem: Operads}
\normalfont 
Let $(\mathbf{C},\mathcal{P}_{\mathcal{S}})$ be a pseudo-tensor $\infty$-category i.e. for every $I\in \mathcal{S},$ 
$$\mathcal{P}_{I}:(\mathbf{C}^{op})^I\times \mathbf{C}\rightarrow \mathbf{Grpd}_{\infty}.$$
Let $F:(\mathbf{C},\mathcal{P}_{\mathcal{S}}^{\mathbf{C}})\rightarrow (\mathbf{D},\mathcal{P}_{\mathcal{S}}^{\mathbf{D}})$ be a pseudo-tensor $\infty$-functor, so in particular we get maps
$$F_I:\mathcal{P}_{I}^{\mathbf{C}}\big(\{c_i\}_{i\in I};c'\big)\rightarrow \mathcal{P}_{I}^{\mathbf{D}}\big(\{F(c_i)\}_{i\in I};F(c')\big),$$
compatible with compositions and preserving identities.

Given a dg-operad $\EuScript{B},$ the $\infty$-category of $\EuScript{B}$-algebras in $(\mathbf{C},\mathcal{P}_{\mathcal{S}})$ is
$$\EuScript{B}-\mathbf{Alg}_{dg}[\mathbf{C},\mathcal{P}_{\mathcal{S}}]:=\mathrm{Fun}\big(\EuScript{B},(\mathbf{C},\mathcal{P}_{\mathcal{S}})\big).$$
In particular, it is given by an object $c\in \mathbf{C}$ and a morphism of dg-operads
$$\beta:\EuScript{B}\rightarrow \mathcal{P}_{\mathcal{S}}(c).$$
\end{rem}

\begin{ex}
    \label{ex: Algebras over operads}
    \normalfont
\noindent{(a)} If $V$ is a (graded, dg) vector space over $\mathbb{C}$, then a $\EuScript{B}$-algebra structure on $V$ is a morphism 
$\EuScript{B}\rightarrow End_V,$ with $End_V=\{End_V(n)\},$ the endomorphism operad of $V$ i.e. its $n$-array operations are multilinear operations 
$$End_V(n):=\mathrm{Hom}_{\mathbb{C}}\big(V^{\otimes_{\mathbb{C}}n},V\big).$$

\noindent{(b)} If $\mathcal{M}$ is a (graded, dg) $\D_X$-module, a $\EuScript{B}^*$-algebra is a morphism 
$$\EuScript{B}\rightarrow End_{\mathcal{M}}^*.$$
The commonly encountered algebraic structures are:
\begin{enumerate}
    \item $\EuScript{B}=\EuScript{C}omm,$ giving commutative $\D_X$-algebras;
    \item $\EuScript{B}=\EuScript{L}ie.$
    In this case, since $\EuScript{L}ie(2)_0
$ is generated by a binary operation $\ell$ of degree $0$, and 
$\mathrm{Hom}_{\mathrm{Operads}_{DG}}(\EuScript{L}ie,End_{\mathcal{M}}^*)$ is equivalent to 
$$\big\{\mu\in \mathcal{P}_{2}^*(\mathcal{M})^0| \mu^{(12)}=-\mu, \mu\circ_{1}\mu+(\mu\circ_1\mu)^{(123)}+(\mu\circ_1\mu)^{(132)}=0\big\}.$$
\item $\EuScript{B}=\EuScript{A}ssoc$-algebra structure in the $*$-operations amounts to a $*$-pairing $\mu_{\mathcal{A}}\in \mathcal{P}_2^*(\{\mathcal{A},\mathcal{A}\},\mathcal{A})$ such that 
$$\mu_{\mathcal{A}}\big(\mu_{\mathcal{A}}(-,-),-\big)=\mu_{\mathcal{A}}\big(-,\mu_{\mathcal{A}})(-,-)\big),\hspace{1mm}\text{ in }\hspace{1mm} \mathcal{P}_{3}^*(\{\mathcal{A},\mathcal{A},\mathcal{A}\};\mathcal{A}).$$
    \end{enumerate}

\noindent{(c)} Changing the pseudo-tensor structure in (b) to chiral operations $\mathcal{P}_{\mathcal{S}}^{ch},$ we get the notion of $\EuScript{B}^{ch}$-algebras
$\EuScript{B}\rightarrow End_{\mathcal{M}}^{ch}.$
So-called chiral Lie algebras studied in \cite{BD} are thus $\EuScript{L}ie^{ch}$-algebras i.e. $\EuScript{L}ie\rightarrow End_{\mathcal{M}}^{ch}.$
\end{ex}
\begin{rem}[Convolution DG-Algebras]
    \label{rem: Convolution}
    \normalfont
    For any Koszul operad $\EuScript{B}$ and graded $\D_X$-module $\mathcal{M}^*$ 
    $$\mathfrak{g}_{\mathcal{M}}^{\EuScript{B}}:=\mathrm{Hom}\big(\EuScript{B}^{c!},End_{\mathcal{M}}^*\big)=\bigoplus_{n}\mathrm{Hom}\big(\EuScript{B}^{c!}(n),End_{\mathcal{M}}^*(n)\big),$$
    is a differential graded Lie algebra with: an \textcolor{red}{internal} grading induced from $\mathcal{M}$, a \textcolor{blue}{cohomological} grading and a \textcolor{orange}{weight} grading induced from $\EuScript{B}^{c!}:$
    $$\mathfrak{g}_{\mathcal{M}}^{\EuScript{B}}=\bigoplus_{\textcolor{blue}{n\geq -1},\textcolor{red}{i\in\mathbb{Z},\textcolor{orange}{d\geq 0}}}\mathfrak{g}_{\textcolor{red}{i}}^{\textcolor{blue}{n},(\textcolor{orange}{d})}.$$
As a dg-Lie algebra there is a MC-equation and one sets
$$\mathrm{Tw}(\mathfrak{g}_{\mathcal{M}}^{\EuScript{B}}):=\{\text{cohom. degree } n=-1 \text{ solutions of } \mathrm{mc}=0\}.$$
Recall the well-known correspondences (e.g. \cite{LoVa}):
$$
\{\EuScript{B}-\text{alg. structures on }\mathcal{M}\}\simeq\mathrm{MC}\big(\mathfrak{g}_{\mathcal{M}}^{\EuScript{B}}\big)
:=\{\text{weight }d=1\text{ elements in }\mathrm{Tw}(\mathfrak{g}_{\mathcal{M}}^{\EuScript{B}})\big\}.$$
\end{rem}

Convolution algebras are especially transparent if one considers not a $\D_X$-module, but a (graded) finite-dimensional vector space $V$ with $\EuScript{B}=\mathcal{A}ssoc,$ the associative operad. Then $\gamma\in \mathrm{MC}(\mathfrak{g}_{V}^{\mathcal{A}ssoc})$ is nothing but an associative product on $V$ and $\mathfrak{g}_{V,\gamma}^{\mathcal{A}ssoc}$ is just the Hochschild complex for the associative algebra $A=(V,\gamma).$

Convolution dg-Lie algebras associated to algebraic structures can be used describe deformations of algebraic structures.

\begin{rmk}
\label{rmk: Deforming Convolution}
\normalfont 

Given $\gamma\in \mathrm{MC}(\mathfrak{g}_{\mathcal{M}}^{\EuScript{B}}),$ there is a twisted dg-Lie algebra $\mathfrak{g}_{\mathcal{M},\gamma}^{\EuScript{B}},$ given by the same underlying dg-Lie algebra but with differential
$d^{\gamma}:=d+[\gamma,-].$
It encodes deformations of the algebraic structure on $\mathcal{M}$, since 
$$\mathrm{MC}(\mathfrak{g}_{\mathcal{M},\gamma}^{\EuScript{B}})\simeq \{\EuScript{B}-\text{alg. structures on }\mathcal{M} \text{ deforming }\gamma\}.$$
\end{rmk}

For any commutative $\D_X$-algebra $\mathcal{A},$
\[
\begin{tikzcd}
    \mathbb{R}\mathcal{H}om(\mathcal{A}^{\otimes 2},\mathcal{A})\arrow[r] & \mathbb{R}\mathcal{P}_{2}^{ch}(\{\mathcal{A},\mathcal{A}\};\mathcal{A})\arrow[d,"\simeq"] \arrow[r] & \mathbb{R}\mathcal{P}_{2}^{*}(\{\mathcal{A},\mathcal{A}\};\mathcal{A})\arrow[d,"\simeq"]
    \\
    & \mathbb{R}\mathcal{P}^{ch,Ran}(\{\mathcal{A}\};\mathcal{A})\arrow[r] & \mathbb{R}\mathcal{P}^{*,Ran}(\{\mathcal{A}\};\mathcal{A}).
\end{tikzcd}
\]

\begin{rem}
    \label{rem: PT-Convolution and Twisting}
    \normalfont
    For a Koszul\footnote{These are a nice class of quadratic operads to which operads $\mathcal{A}ssoc,\mathcal{C}omm,\mathcal{L}ie$ belong.} DG-operad $\EuScript{K}$ let $\EuScript{K}^{c!}$ denote the Koszul dual DG cooperad and  $\Omega\EuScript{K}^{c!}$ the cobar construction. It comes with a quasi-isomorphic map to $\EuScript{K}$ and a $\EuScript{K}_{\infty}$-algebra structure on an object $c$ in a PT-category is a morphism $\Omega\EuScript{K}^{c!}\rightarrow \mathcal{P}^{\mathbf{C}}(c).$
    Every $\EuScript{K}$-algebra structure is in particular a $\EuScript{K}_{\infty}$-structure i.e. there are maps
    $$\Omega\EuScript{K}^{c!}\xrightarrow{qis}\EuScript{K}\rightarrow \mathcal{P}^{\mathbf{C}}(c).$$
By Reminders \ref{rem: Operads} and \ref{rem: Convolution}, we identify
$\mathrm{Hom}_{\mathrm{Operad}^{DG}}(\Omega\EuScript{K}^{c!},\mathcal{P}^{\mathbf{C}}(c)\big),$
with twisting morphisms in the convolution algebra associated to $c$, denoted $\mathcal{C}onv(\EuScript{K},c)$ i.e. the degree $(-1)$-solutions of the MC-equation.
\end{rem}

Given a differential graded $\mathcal{A}[\mathcal{D}_X]$-module $\mathcal{M}^{\bullet},$ a \emph{local $\EuScript{K}_{\infty}$-algebra structure} is a morphism 
$$\Omega\EuScript{K}^{c!}\rightarrow \mathbf{Maps}_{\mathcal{A}}^*(\{\mathcal{M}\},\mathcal{M}).$$

We construct certain dg-Lie algebras that parametrize $\EuScript{A}_{\infty}$-structures and consider the resulting Maurer-Cartan systems in them.

\begin{defn}
    \label{defn: Sh-Ass-* operations}
    \normalfont
    Let $(\mathcal{A},d_{\mathcal{A}})$ be a dg-$\D$-algebra. A \emph{left $\mathcal{A}_{\infty}^*$-module over} $\mathcal{A}$ is a graded left $\D_X$-module $\mathcal{M}$ with a collection of local $*$-operations
    $$\mu=\{\mu_n\in \mathcal{P}_{[n]\sqcup [pt]}^*(\{\mathcal{A},\mathcal{M}\};\mathcal{M})|deg(\mu_n)=1-n\},$$
satisfying  
$$
\sum_{i=1}^n(-1)^{a_1+\ldots+a_{i-1}}\mu_n(a_1,\ldots,d_{\mathcal{A}}a_i,\ldots,a_n\boxtimes m)$$
$$
= \sum_{i=1}^{n-1}(-1)^i\mu_{n-1}(a_1,\ldots,a_{i}\cdot_{\mathcal{A}}a_{i+1},\ldots,a_{n}\boxtimes m)$$
$$-\sum_{p,q\geq 0,p+q=n}(-1)^{q(a_1+\ldots+a_p)+p(q-1)+(p-1)q}\mu_p\big(a_1,\ldots,a_p\boxtimes \mu_q(a_{p+1},\ldots,a_n\boxtimes m)\big),$$
for $a_1,\ldots,a_n\in \mathcal{A}$ and $m\in \mathcal{M},$ with $a\boxtimes m$ an element of the exterior $\D$-module product i.e. the $\mathcal{D}_{X\times X}$-module $\mathcal{A}\boxtimes\mathcal{M}.$
\end{defn}

The $*$-endomorphisms sheaves in $\mathrm{Mod}(\mathcal{D}_X)$ and $\mathrm{Mod}(\mathcal{A}^{\ell}[\mathcal{D}_X])$ are given by
$\mathcal{E}nd_{\mathcal{D}_X}^*(\mathcal{M}):=\mathcal{H}om_{\mathcal{D}_X}(\mathcal{M},\mathcal{M}\otimes\mathcal{D}_X)$ and $\mathcal{E}nd_{\mathcal{A}[\mathcal{D}_X]}^*(\mathcal{M}):=\mathcal{H}om_{\mathcal{A}[\mathcal{D}_X]}(\mathcal{M},\mathcal{M}\otimes\mathcal{A}[\mathcal{D}_X]),$
respectively. 
These spaces provide examples of associative-$*$-algebras as in Example \ref{ex: Algebras over operads} (b-3). The dg-version of the $*$-pairing corresponds to compositions 
$$\mathbb{R}\mathcal{E}nd_{\mathcal{D}_X}^*(\mathcal{M})\boxtimes \mathbb{R}\mathcal{E}nd_{\mathcal{D}_X}^*(\mathcal{M})\rightarrow \Delta_*^{(2)}\mathbb{R}\mathcal{E}nd_{\mathcal{D}_X}^*(\mathcal{M}).$$

It is then possible to adapt the notion of $\mathcal{A}_{\infty}^*$-morphism $f:\mathcal{M}\rightarrow\mathcal{N}$ to the pseudo-tensor context.

\begin{notate}
\label{notate: Strong abuse}
\normalfont 
    By a (strong) abuse of notation, we write an $\mathcal{A}_{\infty}^*$-morphism for a dg-$\D$-algebra $\mathcal{A}$, a left $\mathcal{A}_{\infty}^*$-module $\mathcal{M}^{\bullet}$ (in dg-$\D$-modules) and a dg-$\D$-module $\mathcal{N}^{\bullet}$ over $\mathcal{A}$, by a collection of $\mathcal{A}[\mathcal{D}]$-linear dg-module maps
$$\{f_n:\mathcal{A}^{\boxtimes n}\boxtimes \mathcal{M}\rightarrow \mathcal{N}_*:\mathrm{deg}(f_n)=-n\},$$
where $\mathcal{N}_*$ is the $\mathcal{D}_{X^{n+1}}$-module push-forward. They satisfying the usual compatibility conditions.
\end{notate}

We now classify all morphisms as in \ref{notate: Strong abuse} in a homotopically meaningful way.

\subsection{Classifying the $\mathcal{A}_{\infty}^*$-structures}
We show that the derived $\D$-Quot scheme admits a presentation as a homotopy colimit in the category of dg-$\D$ schemes, though potentially not as a finite-type object. Nevertheless, we prove it is weakly equivalent to a dg-$\D$-manifold of finite type. The argument proceeds via an explicit computation of the relevant homotopy colimits, first in dg-stacks (yielding a dg-scheme), and then through a construction of a $\D$-Postnikov tower (cf. (\ref{eqn: Postnikov})). As a preliminary step, we characterize certain compact derived $\D$-algebras whose cohomologies are finitely-generated $\D$-modules over the classical $\D$-scheme $\mathrm{Spec}_{\D}\mathcal{H}_{\D}^0(\A).$

\begin{prop}
\label{prop: D-fin type result}
Let $X$ be a quasi-projective scheme, or more generally a $\D$-affine smooth $k$-scheme and let $Z$ be a smooth $\D$-scheme of finite type with $\mathbf{Z}$ a dg-$\mathcal{D}$-scheme such that 
$\mathcal{O}_{Z}\rightarrow \mathcal{H}_{\mathcal{D}}^0(\mathcal{O}_{\mathbf{Z}}^{\bullet})$ is a surjective $\D$-algebra morphism and for every $i<0,$ the components of the graded-structure sheaf $\mathcal{H}_{\mathcal{D}}^{-i}(\mathcal{O}_{\mathbf{Z}}^{\bullet})$ are finitely presented $\mathcal{O}_{Z}[\mathcal{D}]$-modules of finite rank e.g. vector $\D$-bundles over $Z.$ Then there exists a dg-$\D$-scheme $\mathbf{W}$ almost of $\D$-finite type that is weakly-equivalent to $\mathbf{Z}$ and whose induced map $\mathbf{W}$ to $Z$ is affine $\D$-schematic.
\end{prop}
\begin{proof}
Since every object has a cofibrant replacement, which can be taken to be semi-free, i.e., of the form $\mathcal{O}_{\mathbf{W}}^{\bullet}=Sym_{\mathcal{O}_Z}^*(V^{\bullet})$, for some 
bounded complex $V^{\bullet}$ of $\D$-bundles over $Z$ in negative degrees, then one may construct the dg-resolution of the structure sheaf. Concretely, the dg-$\D$-scheme is of the form 
$\mathbf{W}=\big(Z,Sym_{\mathcal{O}_Z[\D]}^*(V^{\bullet})\big),$ which comes with a canonical map to $Z$, that is clearly schematic as a morphism of $\D$-schemes, it is constructed explicitly by induction.
Let $\mathcal{O}_{\mathbf{W}_0}^{\bullet}:=\mathcal{O}_Z$. Given $m\leq 0,$ suppose there exists 
$\mathcal{O}_{\mathbf{W}_m}^{\bullet}\rightarrow \mathcal{O}_{\mathbf{Z}}^{\bullet},$ such that 
$\mathcal{H}_{\D}^{i}(\mathcal{O}_{\mathbf{W}_m}^{\bullet})\rightarrow \mathcal{H}_{\D}^{i}(\mathcal{O}_{\mathbf{Z}}^{\bullet}),$ is an isomorphism for each $i>m,$ and moreover satisfies that $\mathcal{H}_{\D}^m(\mathcal{O}_{\mathbf{W}_m}^{\bullet})\rightarrow \mathcal{H}_{\D}^m(\mathcal{O}_{\mathbf{Z}}^{\bullet})$ is a $\D$-epimorphism, between $\mathcal{O}_Z\otimes{\mathcal{O}_X}\D_X$-modules.
Let $Z^m(\mathcal{O}_{\mathbf{W}_m}^{\bullet};\delta)$ be the sub-sheaf of $\mathcal{O}_{\mathbf{W}_m}^m$ consisting of $\delta$-cocycles, with $\delta$ the differential. Then, consider the pull-back of finitely-presented $\mathcal{O}_Z\otimes_{\mathcal{O}_X}\D_X$-modules,
\[
\begin{tikzcd}
\mathcal{M}^m\arrow[d]\arrow[r] & \mathcal{O}_{\mathbf{Z}}^{m-1}
\arrow[d]
\\
Z^m(\mathcal{O}_{\mathbf{W}_m}^{\bullet};\delta)\arrow[r] & \mathcal{O}_{\mathbf{Z}}^m.
\end{tikzcd}
\]
Fix a coherent sheaf $\mathcal{F}$ on $X$. Since it  is projective, any quasi-coherent sheaf may be written as a union of coherent subsheaves. Let $\mathcal{F}'\subseteq \mathcal{F}$ be such a subsheaf. Consider their images under $\D_{\A}$-induction i.e. there exists an injective morphism of $\D_{\A}$-modules, $p_{\infty}^*(\mathcal{F}')\otimes_{\mathcal{O}_X}\D_{\mathcal{A}}\subseteq p_{\infty}^*(\mathcal{F})\otimes_{\mathcal{O}_X}\D_{\A}.$
Then, we may find a sub-sheaf $\widetilde{\M}_m\subset \M^m$ whose image in $Z^m$ is isomorphic to $\M^m.$ Choose a locally free vector $\D$-bundle $\mathcal{N}$ on $Z$, written locally as $\mathcal{N}\simeq p_{\infty}^*(\mathcal{E}_X)\otimes \mathcal{O}_Z[\D],$ for some coherent sheaf $\mathcal{E}_X$ on $X.$ It may be chosen such that there exists a surjective morphism $\mathcal{N}\rightarrow \widetilde{\M}_m.$
Then let $\mathcal{O}_{\mathbf{W}_{m'}}^{\bullet}$ denote the dg-structure sheaf obtained from $\mathcal{O}_{\mathbf{W}_m}^{\bullet}$ by freely adjoining the generators of $\mathcal{N}$ in degree $(m-1),$ with the differential determined by the map $\mathcal{N}\rightarrow Z^m(\mathcal{O}_{\mathbf{W}_m}^m;\delta).$ Thus, there exists a morphism $\mathcal{O}_{\mathbf{W}_{m'}}^{\bullet}\rightarrow \mathcal{O}_{\mathbf{Z}}^{\bullet}$ together with a morphism 
$\mathcal{H}_{\D}^{\geq m}(\mathcal{O}_{\mathbf{W}_{m'}}^{\bullet})\rightarrow \mathcal{H}_{\D}^{\geq m}(\mathcal{O}_{\mathbf{Z}}^{\bullet}),$ which is an isomorphism. Then, since $\mathcal{H}_{\D}^{m-1}(\mathcal{O}_{\mathbf{Z}}^{\bullet})$ is a finite-presentation $\mathcal{O}_Z[\D_X]$-module, we repeat the above argument, used to construct $\M^m.$ That is, there exists a finite-presentation $\D_{\A}$-module $\mathcal{K}\subseteq Z^{m-1}$ together with an induced $\mathcal{O}_Z[\D]$-module surjection $\mathcal{K}\rightarrow \mathcal{H}_{\D}^{m-1}(\mathcal{O}_{\mathbf{Z}}^{\bullet}).$
Choose a subsheaf $\mathcal{K}_0$ and a surjection $\mathcal{K}_0\rightarrow \mathcal{K},$ we define 
$\mathcal{O}_{\mathbf{W}_{m-1}}^{\bullet})$ by considering $\mathcal{O}_{\mathbf{W}_{m'}}^{\bullet}$ and adjoining $\mathcal{K}_0[m-1],$ i.e. generators in degrees $(m-1).$ Thus, there exists a canonical morphism $\ell_{m-1}:\mathcal{O}_{\mathbf{W}_{m-1}}^{\bullet})\rightarrow \mathcal{O}_{\mathbf{Z}}^{\bullet},$ such that $\mathcal{H}_{\D}^{\geq m}(\mathcal{O}_{\mathbf{W}_{m-1}}^{\bullet})\simeq \mathcal{H}_{\D}^{\geq m}(\mathcal{O}_{\mathbf{Z}}^{\bullet})$ is an isomorphism, and the induced morphism $\mathcal{H}_{\D}^{m-1}(\ell_{m-1})$ is a surjection. This completes the inductive step and the rest is clear.
\end{proof}
Consider for all $t\geq a$ and arity $n\geq 1$ in the multilinear operations, the convolution dg $*$-algebra structure as in \ref{rem: Convolution} in the pseudo-tensor structure of $\mathcal{A}^{\ell}$-modules in $\mathbb{Z}_{+}$-graded perfect $\D_X$-modules $\mathcal{P}^*$ restricted to \emph{degree $0$} morphisms:
$$\mathfrak{g}_{\geq a;t}^{n}:=\mathcal{P}_{[n]\sqcup *}^{*}(\{\mathcal{A},\mathcal{N}_{\geq a}\};\mathcal{N}_t\big)\oplus \mathcal{P}_{[n]\sqcup\{*\}}^*(\{\mathcal{A},\mathcal{N}_{\geq a}\};\mathcal{C}h_t\big).$$
We recall that the $*$-pairing and its $Ran_X$-analog,
$$\mathcal{P}^{*,Ran}(\{\mathcal{N}\},\mathcal{M})^0=\mathbb{R}\mathcal{H}om_{\mathcal{D}_{Ran_X}}\big(\mathcal{N}^{\otimes^{\star n}},\mathcal{M}\big)^{0},$$
via the $\otimes^{\star}$-tensor product \cite[Sect 2.2]{FG}. The corresponding operations in $\mathcal{A}-\mathbf{Mod}(\mathcal{D}_{Ran_X})$ we must use, are formally defined using the pointed Ran space $Ran_{X,*}$, but for our purposes, we need only know that they are described by 
$$\mathcal{P}_{[n]\sqcup *}^{*}(\{\mathcal{A},\mathcal{N}\};\mathcal{M})^0:=\mathbb{R}\mathcal{H}om_{\mathcal{A}\otimes\mathcal{D}_{Ran_X}}(\mathcal{A}^{\otimes^* N}\otimes \mathcal{N},\mathcal{M})^0,$$
where notation $[n]\sqcup\{*\}$ means that for every $n\geq 0$ we have a multi-operation on families $\{\mathcal{A},\mathcal{N}\},$ of a $\D_X$-algebra $\mathcal{A}$ and a $\mathcal{A}[\mathcal{D}_X]$-module $\mathcal{N},$ where we keep the arity of the second argument (corresponding to the module entry) fixed along a given stratum of $Ran_X$ corresponding to $\{*\}.$ Restrict again to degree zero graded morphisms.
\begin{prop}
\label{prop: DGLie*Alg}
The object $\mathfrak{g}_{a,t}^{\bullet}:=\bigoplus_{a\leq s\leq t,n\geq 1}\mathfrak{g}_{\geq a;s}^{n}$ is a $\mathbb{Z}_+$-graded $Lie^*$-algebra object in $\mathcal{A}$-modules in $\mathbf{Mod}(\mathcal{D}_{Ran_X}).$ In particular, there is a canonical degree zero canonical $*$-pairing, 
    $$\mu\in \mathcal{P}_{[2]}^{*,Ran}\big(\{\mathfrak{g}_{\geq a;t_2}^{\bullet},\mathfrak{g}_{\geq a;t_1}^{\bullet}\},\mathfrak{g}_{a;t_2}^{\bullet}\big)^{0},$$
    induced by composition.
\end{prop}
Remark that all $\D$-modules and algebras are supported on the main diagonal i.e. all objects of $\mathcal{D}_{Ran_X}$ we consider are coming from $X\subset Ran_X$ by push-forward.

The constructions in \cite[Sect. 2.1, pg. 6]{BKSY2}, follow in our case to produce differential graded $Lie^*$-algebra objects. By projection onto the direct summands in $\mathfrak{g}^{\bullet},$ as given in Proposition \ref{prop: DGLie*Alg}, we obtain a sequence
\begin{equation}
\label{eqn: DGLA sequence}
\cdots\rightarrow \mathfrak{g}_{a,a+2}^{\bullet}\rightarrow \mathfrak{g}_{a,a+1}^{\bullet}\rightarrow \mathfrak{g}_{a}^{\bullet}.
\end{equation}
Over the fixed $\D$-scheme $\mathrm{Spec}_{\mathcal{D}}(\mathcal{A})$, each term in (\ref{eqn: DGLA sequence}) determines a formal derived $\D$-stack
\begin{equation}
\label{eqn: Y D stack}\mathcal{Y}_{[a,a+2]}:=\underline{\mathbf{Spec}}_{\mathcal{D}}\big(CE^{\bullet,\otimes^{\star}}(\mathfrak{g}_{a,t}^{\bullet}[1])^{\circ}\big).
\end{equation}
In other words, the dg-$\D$-algebra of functions $\mathcal{O}_{[a,a+2]}$ on (\ref{eqn: Y D stack}) is generated by the local dual of the direct summands $\mathfrak{g}^{\bullet}.$ 
We end up with a sequence of formal derived $\D$-stacks,
\begin{equation}
    \label{eqn: Formal derived stack sequence}\ldots\rightarrow \mathcal{Y}_{[a,a+2]}\rightarrow \mathcal{Y}_{[a,a+1]}\rightarrow \mathcal{Y}_{a},
    \end{equation}
or dually, a sequence of cocommutative co-algebras in $\D$-modules i.e. commutative (factorization) co-algebras $\mathcal{B}_{[a,t]}^{\bullet}$ generated by the $(\mathfrak{g}_{a,t}^{\bullet}[1])^{\circ}.$ Since $\mathfrak{g}^{\bullet}$ are $\D$-finitely generated with coherent components we obtain that $\mathcal{B}_{[a,t]}^{\bullet}$ is a finitely presented $\D$-factorization algebra object supported on the main diagonal and thus a finite type $\D$-algebra on $X.$
\begin{prop}
\label{prop: colim of algebra sequence}
   Consider the sequence \emph{(\ref{eqn: DGLA sequence})}. It induces a sequence
   $$\ldots\leftarrow\mathcal{B}_{[a,a+2]}^{\bullet}\leftarrow \mathcal{B}_{[a,a+1]}^{\bullet}\leftarrow \mathcal{B}_{a}^{\bullet},\hspace{1mm} in\hspace{2mm} \mathbf{CAlg}_{X}(\mathcal{D}_X)\subset \mathbf{FAlg}_X(\mathcal{D}_X)^{comm},$$ of freely generated $\D_X$-algebras by the local duals $\mathfrak{g}_{a,t}^{\bullet}[1]^{\circ},$ which are finitely-generated with finite dimensional graded components given by $\mathcal{A}[\mathcal{D}]$-modules of finite rank. The homotopy colimit of $\D$-algebras defines a derived $\D$-stack,  $$\EQ_{a}^{tot}:=\mathbf{Spec}_{\mathcal{D}}\big(\ldots\leftarrow\mathcal{B}_{[a,a+2]}^{\bullet}\leftarrow \mathcal{B}_{[a,a+1]}^{\bullet}\leftarrow \mathcal{B}_{a}^{\bullet}\big).$$
\end{prop}
We prove Proposition \ref{prop: colim of algebra sequence} by adapting the strategy of \cite{BKSY2} and using Kashiwara's theorem for the diagonal embedding together with Koszul duality. Alternatively $\EQ_a^{tot}$ taken in Proposition \ref{prop: colim of algebra sequence} agrees with the corresponding homotopy limit taken in the category of (formal) derived $\D$-stacks; however, it is not of finite type in the $\D$-sense, but rather of the pro/ind-affine type. We overcome this problem by adapting the method in \cite[Prop 4. and Thm 3. of Sect.4]{BKSY2}. 

\subsubsection{Injectivity locus and the quotient}
\label{sssec: Injectivity locus and the quotient}
Consider $a_1>a_0\geq a,$ and the induced morphism 
$$\mathfrak{g}_{a_0;t}^{\bullet}:=\mathcal{P}_{[n]\sqcup *}^{*}(\{\mathcal{A},\mathcal{N}_{a_0}\};\mathcal{N}_{t}\big)\oplus \mathcal{P}_{[n-1]\sqcup\{*\}}^*(\{\mathcal{A},\mathcal{N}_{a_0}\};\mathcal{C}h_t\big)\rightarrow $$
$$\rightarrow \mathcal{P}_{[n]\sqcup *}^{*}(\{\mathcal{A},\mathcal{N}_{a_1 }\};\mathcal{N}_t\big)\oplus \mathcal{P}_{[n-1]\sqcup\{*\}}^*(\{\mathcal{A},\mathcal{N}_{a_1}\};\mathcal{C}h_t\big),$$
which acts by mapping $\mathfrak{g}_{<a_1;s}^k$ to zero for each $k\geq 1,$ and for every $s\geq a_0.$ 
One has an induced sequence of DG-$\mathcal{D}_X$-modules,
$$\mathfrak{h}_{a_0,a_1;t}^{\bullet}\rightarrow\mathfrak{g}_{a_0;t}^{\bullet}\rightarrow \mathfrak{g}_{a_1;t}^{\bullet},$$
with $\mathfrak{h}$ the fiber.

For every $l\geq 1$ and for each $t \geq s$ there is a corresponding right-action on the direct sum 
$$\mathcal{P}_{k\sqcup *}^0(\{\mathcal{R}_+,\mathcal{N}_s\};\mathcal{C}h_{s+1,t})\oplus \mathcal{P}_{k-1\sqcup *}^0(\{\mathcal{R}_+,\mathcal{N}_s\};\mathcal{M}_{s,t}).$$
This generates an action on $\D$-prestacks $\mathcal{Y}_{\alpha,t}.$
\begin{prop}
\label{prop: DG Decomp}
    For every $a_1>a_0\geq a,$ there is a natural morphism between affine DG-$\mathcal{D}$-schemes 
$\mathcal{Y}_{a_0,t}\rightarrow \mathcal{Y}_{a_1;t}.$
More generally, for any $a_0<a_1<\ldots<a_j$ with $t>a_j,$ there are decompositions of graded $\mathcal{D}$-modules 
$$\mathfrak{g}_{a_0;t}^{\sharp}\simeq \mathfrak{g}_{a_j;t}^{\sharp}\oplus \bigoplus_{0\leq i<j}\mathfrak{h}_{a_i,a_{i+1};t}^{\sharp},$$
with induced morphisms 
$f_{a_0,a_j;t}:\mathcal{Y}_{a_0;t}\rightarrow \mathcal{Y}_{a_j;t},$ such that 
$$\mathcal{O}_{\EQ_{a_0;t}}\simeq f^{-1}_{a_0,a_j;t}(\mathcal{O}_{\mathcal{Y}_{a_j;y}})\oplus \bigoplus_{0\leq i<j}f_{a_0;a_j,t}^{-1}\mathcal{J}_{a_i,a_{i+1};t}.$$
\end{prop}
\begin{proof}
Follows immediately and from the fact that $\mathcal{Y}_{b;s}$ as in (\ref{eqn: Y D stack}) is by construction an affine $\mathcal{D}$-space represented by a cocommutative co-algebra in $\mathcal{D}$-modules generated by (Koszul dual of) $\mathfrak{g}_{b;s}^{\bullet}[1].$
\end{proof}

Consider now the sub-space
$$\mathcal{H}om(\mathcal{N}_s,\mathcal{C}h_s)\subseteq \mathfrak{g}_{\geq a;t}^1,\hspace{2mm} s\geq a,$$
consisting of morphisms which correspond to sub-modules i.e. they are injective $\D$-module morphisms. This turns out to be an open condition obtained by imposing a maximality of rank (recall each graded component is finitely rank vector $\D$-bundle, thus posses a well-defined rank). In turn this defines differential graded sub-manifolds given by the DG-$\mathcal{D}$-subschemes 
\begin{equation}
\label{eqn: W-Subschemes}
\mathcal{W}_{[a,t]}\hookrightarrow \mathcal{Y}_{a,t},\forall t\geq a.
\end{equation}
Consequently, we obtain a sequence induced from (\ref{eqn: Formal derived stack sequence}),
denoted 
$$\ldots\rightarrow \mathcal{W}_{[a,a+2]}\rightarrow\mathcal{W}_{[a,a+1]}\rightarrow \mathcal{W}_{a}.$$
The morphisms in this sequence are not necessarily $\D$-affine. 
\begin{prop}
    For every $t\geq b,$ consider the affine dg-$\mathcal{D}$-scheme, defined by the fiber product
    $$\mathcal{U}_{[a,b];t}\simeq \mathcal{W}_{[a,b]}\times_{\mathcal{Y}_{b;t}}^h\mathcal{Y}_{a;t},$$
    via the natural fibration $\mathcal{Y}_{a;t}\rightarrow \mathcal{Y}_{b;t}$ and the open embedding \emph{(\ref{eqn: W-Subschemes})}.
    For every $t\geq b,$ one has the induced morphism of differential graded commutative $\mathcal{D}$-algebras
    $$\mathcal{O}_{\mathcal{W}_{[a,t]}}\rightarrow \mathcal{O}_{\mathcal{U}_{[a,b];t}},$$
    is a quasi-isomorphism.
\end{prop}

Among the space $\mathcal{U}_{[a,b];t}$ look at the open-subscheme formed by the geometric quotient and let 
$$\mathcal{V}_{[a,t]}^0\subseteq \mathcal{U}_{[a,b];t}^0,$$
denote the open subset given by requiring the additional condition: 
\begin{equation}
    \label{eqn: Surjectivity condition}
\bullet\text{ for every } t\geq s>s'\geq b,\text{ consider only surjections in } P_{[1]\sqcup *}(\mathcal{R}_{s-s'},\mathcal{N}_{s'};\mathcal{N}_s).
\end{equation}
Condition (\ref{eqn: Surjectivity condition}) leads to a well-defined object $\mathcal{V}_{[a,t]}$ defined as the corresponding open dg-$\D$-submanifold given by restriction of the structure sheaf.

\subsubsection{Computing a colimit}
\label{sssec: Computing a colimit}
We use \cite[Prop. 4.3.3-4.3.4]{CFK}, which adapted to our context looks as follows. We may choose 
$\alpha_0<\alpha_1<\ldots, \beta_0<\beta_1<\ldots,t_0<t_1<\ldots,$ such that 
\begin{enumerate}
    \item For each $j\geq 0,$ and for every $t\geq t_j$ the natural map $\mathcal{U}_{[\alpha_j,\beta_j],t}/\mathsf{G}_{\beta_j+1,t}\rightarrow \mathcal{U}_{[\alpha_j,\beta_j],t_j}/\mathsf{G}_{\beta_j+1,t_j},$ induces an isomorphism on classical schemes and a quasi-isomorphism of tangent complexes in degrees $[0,\ell+j+1].$
\end{enumerate}
Thus, there is for every $j\geq 0$ a sequence of weak-equivalences 
$$\mathsf{P}_{\leq \ell+j}(\mathcal{V}_{\alpha_j,t_j})\xleftarrow{\sim} \mathsf{P}_{\leq \ell+j}(\mathcal{V}_{\alpha_j,t_j+1})\xleftarrow{\sim} \dots,$$
where for every $r>0$ we recall that $\mathsf{P}_{\leq r}$ denotes the $r$-th truncation (\ref{eqn: Postnikov}). In particular, it is determined by a sheaf of dg-ideals coming from embedding $\mathsf{P}_{\leq r}(\mathcal{V}_{\alpha,t})\hookrightarrow \mathcal{V}_{\alpha,t}$.

We also get similar maps for each $j\geq 0,$ from $\mathcal{V}_{\alpha_j,t_j+1}\rightarrow \mathcal{V}_{\alpha_j+1,t_j+1}$. These yield weak-equivalences on the $\mathsf{P}_{\leq \ell+j}$'s and in fact we obtain the following infinite diagram of Postnikov truncations:
\begin{equation}
    \label{eqn: Diagram 1}
\begin{tikzcd}
\cdots \rightarrow \mathsf{P}_{\leq \ell}(\mathcal{V}_{\alpha_0,t_2})\arrow[d]\arrow[r,"\sim"] & \mathsf{P}_{\leq \ell}(\mathcal{V}_{\alpha_0,t_1})\arrow[d] \arrow[r,"\sim"] & \mathsf{P}_{\leq \ell}(\mathcal{V}_{\alpha_0,t_0})
\\
\cdots \rightarrow \mathsf{P}_{\leq \ell+1}(\mathcal{V}_{\alpha_1,t_2})\arrow[d]\arrow[r,"\sim"] & \mathsf{P}_{\leq \ell+1}(\mathcal{V}_{\alpha_1,t_1})  & 
\\
\cdots\rightarrow \mathsf{P}_{\leq \ell+2}(\mathcal{V}_{\alpha_2,t_2}) & & 
\end{tikzcd}
\end{equation}

Diagram (\ref{eqn: Diagram 1}) can be replaced by a weakly-equivalent one for which taking the homotopy limits of the rows followed by the homotopy colimit of the entire diagram makes sense. 
\begin{prop}
\label{prop: Z-diagram}
For every $t\geq t_0$, put $\mathbf{Z}_{0;\alpha_0,t}:=\mathsf{P}_{\leq \ell}\big(\mathcal{U}_{[\alpha_0,\beta_0],t}/\mathsf{G}\big).$ Then, there is a natural diagram, 
    \begin{equation}
        \label{eqn: Z-diagram}
          \begin{tikzcd}
\cdots \rightarrow \mathbf{Z}_{0;\alpha_0,t_2}\arrow[d]\arrow[r,"\sim"] & \mathbf{Z}_{0;\alpha_0,t_1}\arrow[d] \arrow[r,"\sim"] & \mathbf{Z}_{0;\alpha_0,t_0}
\\
\cdots \rightarrow \mathbf{Z}_{1;\alpha_0,t_1}\arrow[d]\arrow[r,"\sim"] & \mathbf{Z}_{1;\alpha_0,t_1}  & 
\\
\mathbf{Z}_{2;\alpha_0,t_1} & &  
        \end{tikzcd}
    \end{equation}
   which is moreover weakly-equivalent to \emph{(\ref{eqn: Diagram 1})}.
\end{prop}
\begin{proof}
By Proposition \ref{prop: D-fin type result}, the homotopy limit in dg-$\D$-schemes of ind-finite type is the derived $\D$-Quot scheme. We can then show it is weakly equivalent to a dg-$\D$-manifold of finite $\D$-type using the structure sheaf decomposotion of Proposition \ref{prop: DG Decomp}, for each $j>0$ and for all $t\geq t_j.$
Restriction to open dg-$\D$-subschemes and taking quotients, for all $j>0,$ we obtain sheaves of non-positively graded
$\mathcal{O}_{\mathcal{U}_{[\alpha_0,\beta_0],t}/\mathsf{G}}^0[\mathcal{D}]$-submodules, of the form 
$$\bigoplus_{0\leq i<j}\mathcal{J}_{\alpha_i,\alpha_{i+1};t}^*\hookrightarrow \mathcal{O}_{\mathcal{U}_{[\alpha_0,\beta_0],t}/\mathsf{G}}^{\bullet}.$$
Consider for each $j<0$ and $t\geq t_j$ the truncation 
$\mathsf{P}_{\leq \ell+j}\big(\mathcal{U}_{[\alpha_0,\beta_0],t}/\mathsf{G}\big),$ which comes with a closed dg-$\D$-subscheme
$$\mathbf{Z}_{j;\alpha_0,t}\hookrightarrow \mathsf{P}_{\leq \ell+j}\big(\mathcal{U}_{[\alpha_0,\beta_0],t}/\mathsf{G}\big),$$
associated with the dg-$\D$-ideal sheaf generated by
$$\bigoplus_{0\leq i<j}\bigoplus_{m<-\ell-i}\mathcal{J}_{\alpha_i,\alpha_{i+1},t}^m.$$
\end{proof}
With this presentation of the $\D$-Postnikov diagram, one may prove representability by taking a suitable homotopy colimit.

\begin{thm}
\label{thm: Finite--type dg-D-Quot}
The derived $\D$-dg Quot scheme is a dg-$\D$-manifold homotopically of finite $\D$-type.
\end{thm}
\begin{proof}
Consider Proposition \ref{prop: Z-diagram}. Take the homotopy limit of each row in (\ref{eqn: Z-diagram}), denoting the result by 
$$\mathbf{Z}_{j;\alpha_0,\infty}:=\mathrm{holim}_{j\geq 0} \mathbf{Z}_{j;\alpha_0,t_j}.$$
We obtain a diagram
$$\mathbf{Z}_{0;\alpha_0,\infty}\rightarrow \mathbf{Z}_{1;\alpha_0,\infty}\rightarrow \mathbf{Z}_{2;\alpha_0,\infty}\rightarrow\cdots.$$
Comparing with (\ref{eqn: Diagram 1}), it gives a $\D$-geometric Postnikov tower (c.f sequence \ref{eqn: Postnikov}). Therefore, we may take the homotopy colimit in dg-$\D$-schemes:
$$\mathbf{Z}_{\infty;\alpha_0;\infty}\simeq\mathrm{hocolim}_{j\geq 0}\mathbf{Z}_{j;\alpha_0,\infty}.$$
Obviously it is not of finite type. Nevertheless, one can notice that each object $\mathbf{Z}_{j;\alpha_0,t_j}$ in (\ref{eqn: Z-diagram}) has the property that $\mathcal{O}_{\mathbf{Gr}_{[\alpha_0,\beta_0]}}$ surjects onto each $\mathcal{H}_{\mathcal{D}}^{0}(\mathcal{O}_{\mathbf{Z}_{j;\alpha_0,t_j}^{\bullet})}$ and for every $m<0$ the cohomology sheaves
$\mathcal{H}_{\mathcal{D}}^{m}(\mathcal{O}_{\mathbf{Z}_{j;\alpha_0,t_j}^{\bullet})}$ are coherent $\mathcal{O}[\mathcal{D}]$-modules. 
Thus the same is true upon taking the homotopy limit and one concludes by applying Proposition \ref{prop: D-fin type result}.
\end{proof}

\end{document}